\documentclass{amsart}
\usepackage{amsmath}
\usepackage{framed,comment,enumerate}
\usepackage[pdftex]{graphicx}
\usepackage{epstopdf}
\usepackage{xcolor, framed}
\usepackage{color}

\def\R {\mathbb{R}}
\def\eps{\varepsilon}
\def\OP{obstacle problem }
\def\FB{free boundary }

\def\Qua{\mathcal{Q}}
\def\UQua{\mathcal{UQ}}

\def\Sph{\mathbb{S}^{d-1}}

\def\SPE{\mathcal{S}(p,\eps,1)}

\def\PosS{\{u>0\}}
\def\Singu{\Sigma(u)}

\def\Dee{D_{ee}}

\def\hh{\hat{h}}
\def\hu{\hat{u}}
\def\hO{\hat{O}}

\def\minx{\underline{x}'}

\newtheorem{prop}{Proposition}[section]
\newtheorem{thm}{Theorem}[section]

\newtheorem{lem}{Lemma}[section]
\theoremstyle{definition}
\newtheorem{defi}{Definition}[section]
\newtheorem{rem}{Remark}[section]
\numberwithin{equation}{section}

\title[Regularity of the singular set]{Regularity of the singular set in the fully nonlinear obstacle problem } 

\author{Ovidiu Savin}
\address{Department of Mathematics,	Columbia University, New York, USA}
\email{savin@math.columbia.edu}

\author{Hui Yu}
\address{Department of Mathematics,	Columbia University, New York, USA}
\email{ huiyu@math.columbia.edu}
\thanks{O.~S.~is supported by  NSF grant DMS-1500438.}

\begin{document}

\begin{abstract}
For the \OP involving a convex fully nonlinear elliptic operator, we show that the singular set in the \FB stratifies. The top stratum is locally covered by a $C^{1,\alpha}$-manifold, and the lower strata are covered by $C^{1,\log^\eps}$-manifolds. This recovers some of the recent regularity results due to Colombo-Spolaor-Velichkov \cite{CSV} and Figalli-Serra \cite{FSe} when the operator is the Laplacian.\end{abstract}

\maketitle

%\tableofcontents
%%%%%%%%%%%%%%%%%%%%%%%%%%%%%%%%%%%%%%%%%%%%%%%%%%%%%%%%%%%%%%%%%%%%%%%%%%%%%%%%%%%%%%%%%%%%%%%%%%%%%%
\section{Introduction}
The \textit{classical \OP}describes the equilibrium shape of an elastic membrane being pushed towards an impenetrable barrier. In its most basic form, the height of the membrane satisfies the following equations
\begin{equation*}
\Delta u=\chi_{\PosS} \text{ and }u\ge 0 \text{ in $\Omega$.}
\end{equation*}Here $\Omega$ is a given domain in $\R^d$, and $\chi_E$ denotes the characteristic function of the set $E$. The right-hand side of the first equation has a jump across the a priori unknown interface $\partial\PosS$, often called the \textit{free boundary}. 

Apart from its various industrial applications, many ideas and techniques developed for the  classical \OP have been crucial in the study of other \FB problems. In this sense, the classical \OP is the prototypical \FB problem.  As a result, it has been studied extensively during the past few decades. For many applications of the classical \OP and some related problems, see  Petrosyan-Shahgholian-Uraltseva \cite{PSU} and Ros-Oton \cite{R}.

As already observed by Br\'ezis-Kinderlehrer \cite{BK}, the solution $u$ enjoys the optimal $C^{1,1}_{loc}$ regularity. The interesting problem is to understand the regularity of the \FB $\partial\PosS.$ In that direction, Sakai first gave some results for the planar case in  \cite{Sak1, Sak2}. The theory in higher dimensions was developed by Caffarelli in \cite{C1, C2}, where he showed that around points on $\partial\PosS$ the solution $u$ has two possible types of behaviour. Either it behaves like a \textit{half-space solution}, $\frac{1}{2}\max\{x\cdot e,0\}^2$ for some $e\in\mathbb{S}^{d-1},$ or it behaves like a quadratic polynomial $\frac{1}{2}x\cdot Ax$ for some non-negative matrix $A$.

The points where the solution behaves like half-space solutions are called \textit{regular points}. Near such points, Caffarelli showed that the \FB is an analytic hypersurface \cite{C1}. His method is sufficiently robust that it has been adapted for regular points in many other problems, including the thin-obstacle problem in Athanasopoulos-Caffarelli-Salsa \cite{ACS}, the obstacle problem for integro-differential operators in Caffarelli-Salsa-Silvestre \cite{CSS}, the \OP for fully nonlinear operators in Lee \cite{L}, the \OP for fully nonlinear non-local operators in Caffarelli-Serra-Ros-Oton \cite{CSR}, and a very general class of unconstrained \FB problems in Figalli-Shahgholian \cite{FSh} and Indrei-Minne \cite{IM}.

The points where the solution behaves like quadratic polynomials are called \textit{singular points}. As shown by Schaeffer \cite{Sch}, the \FB can form cusps near these points. Nevertheless, certain structural results can be established for singular points. 

To be precise, let $\frac{1}{2}x\cdot A_{x_0}x$ denote the polynomial modelling the behaviour of $u$ around a singular point $x_0.$ Depending on the dimension of the kernel of $A_{x_0}$, the collection of singular points can be further divided into $d$ classes (\textit{strata}), the $k$th stratum being $$\Sigma^k(u)=\{x_0|\text{ $x_0$ is a singular point with $\dim ker(A_{x_0})=k$}\}.$$ 
The structural theorem by Caffarelli says that $\Sigma^k$ is locally covered by $C^{1}$-manifolds of dimension $k$ \cite{C2}.  His proof was based on the Alt-Caffarelli-Friedman formula in \cite{ACF}. An alternative proof was later found by Monneau \cite{M}, using the monotonicity formula bearing his name. 

Recently there has been quite some interest in improving this result. In two dimension, Weiss improved the regularity of the manifolds to $C^{1,\alpha}$  by introducing the  Weiss monotonicity formula \cite{W}. Based on the same formula, Colombo-Spolaor-Velichkov \cite{CSV} showed that in higher dimensions the manifolds are $C^{1,\log^\eps}$. The best result so far is in Figalli-Serra \cite{FSe}. By applying Almgren's monotonicity formula \cite{Alm}, they improved $C^{1,\log^\eps}$ to $C^{1,\alpha}$ for the manifolds covering the top stratum $\Sigma^{d-1}(u)$. They also showed that each stratum can be further divided into a `good' part and a `bad' part, where the former is covered by $C^{1,1}$ manifolds, and the latter is of lower dimension.

Despite these exciting new results, almost nothing is known about singular points for obstacle problems involving operators other than the Laplacian.  Comparing with the robust argument for regular points, all developments on singular points depend on various  monotonicity formulae. These are powerful but restricted, in the sense that they are not expected to hold for nonlinear operators or even for linear operators with coefficients of low regularity.  This same obstruction lies behind the lack of understanding of singular points in many other \FB problems. Consequently, it is important to develop new tools when monotonicity formulae are not available.

In this work, we develop a method for the study of singular points without relying on monotonicity formulae. In particular, this method works for the following \OP involving a convex fully nonlinear elliptic operator $F$ whose derivatives are H\"older continuous:\begin{equation}\label{OP}
\begin{cases}F(D^2u)=\chi_{\PosS}, &\\u\ge0,&
\end{cases} \text{ in $\Omega$.}
\end{equation} Here $\Omega$ is a domain in $\R^d$, and the solution $u$ is understood in the viscosity sense, see \cite{CC,L}. For a given boundary data, the solution $u$ is unique and can be obtained either as the least nonnegative supersolution to the the equation $F(D^2 u) \le 1$, or as the largest subsolution to $F(D^2 u) \ge \chi_{\{u>0\}}$.

Even for the case when $F$ is the Laplacian, our method is interesting as it provides a new approach to the regularity of the singular set.  At first reading, it might relieve many technical complications if the reader takes $F$ to be the Laplacian. 

For the singular points on the \FB $\partial\PosS$, our main result reads

\begin{thm}\label{MainResult}
Let $u$ be a solution to \eqref{OP}. 

For $k=0,1,\dots,d-2$, the $k$-th stratum of the singular points, $\Sigma^k(u)$, is locally covered by a $k$-dimensional $C^{1,\log^\eps}$-manifold.  The top stratum, $\Sigma^{d-1}(u)$, is locally covered by a $(d-1)$-dimensional $C^{1,\alpha}$-manifold. \end{thm}  
Theorem \ref{MainResult} states that the singular set of the free boundary in the nonlinear obstacle problem setting enjoys similar regularity properties as in the linear case. The methods developed here rely on linearization techniques, and the hypothesis that $F \in C^1$ is essential in our analysis.

 Let us briefly recall the strategy when the operator is the Laplacian. For each point $x_0$ in the singular set, we study the resclaings $u_{x_0,r}(\cdot)=\frac{u(r\cdot+x_0)}{r^2}$ as $r\to 0$. \textit{Up to a subsequence}, they converge to a quadratic polynomial, called the blow-up profile at $x_0$. When the operator is the Laplacian,  this polynomial is \textit{unique} in the sense that it is independent of the subsequence $r\to 0.$ It models the behaviour of our solution `at the point $x_0$'. A uniform rate of convergence allows the comparison of  blow-up profiles at different points. This gives the desired regularity of the covering manifolds.
 
Up to now, however, even the proof for the uniqueness of the blow-up profile requires monotonicity formulae. Due to the unstable nature of singular points, it is not obvious that the solution  cannot behave like completely different polynomials at different scales. This can be ruled out by monotonicity formulae. Once the solution is close to a parabola at a certain scale, a monotone quantity shows that the solution stays close to the same parabola at all smaller scales, leading to uniqueness of the blow-up profile. 

Since no monotonicity formula is expected for our problem, we do not have access to the behavior of $u$ at all small scales. Instead, we proceed using an iterative scheme. Suppose the  solution is very close to a parabola in $B_1$, we need to show that for some $\rho<1$, it is even closer to a similar parabola in $B_\rho$. Iterating this argument gives a rate of convergence to the blow-up profile, which in particular gives its uniqueness. Such scheme has been applied to study regularity of solutions of elliptic equations \cite{Sa} as well as regular points along free boundaries \cite{D}. To our knowledge this is the first time it has been applied to singular points along free boundaries. 

To be precise, suppose that $0$ is a singular point along $\partial\PosS$, and that $u$ is very close to a parabola $p$   in $B_1$, in the sense that  $$|u-p|<\eps  \text{ in $B_1$}$$ for some small $\eps$.  Our goal is to show that in $B_\rho$, the solution $u$ can be better approximated.  It is natural to look at the normalized solution $$\hat{u}=\frac{1}{\eps}(u-p),$$ which solves an obstacle problem with $\hat{O}=-\frac{1}{\eps}p$ as the obstacle. Assume that $p$ takes the form $$p=\frac{1}{2}\sum_{1\le j\le d} a_jx_j^2$$ with the coefficients satisfying  \begin{equation}\label{OrderingOfCoefficients}a_1\ge a_2\ge\dots\ge a_d,\text{ and } a_k\gg\eps,\end{equation} then the contact set between $\hat{u}$ and the obstacle concentrates around the subspace $$\{x_1=x_2=\dots =x_{k}=0\}.$$ 

From here we need to separate two cases depending on the dimension of this subspace.

When $k=1$, this subspace is of codimension $1$. In the limit as $\eps\to 0$, $\hat{u}$ effectively solves \textit{the thin obstacle problem} with $0$ as the obstacle along $\{x_1=0\}$. Let $\bar{u}$ denote the solution to this problem. After developing new technical tools concerning the directional monotonicity and convexity of solutions, we can show that $\bar{u}$ is $C^2$ at the origin, and the second-order Taylor polynomial of $\bar{u}$ gives the approximation of $u$ in $B_\rho$ with an error of the order $(1-\beta)\eps\rho^2$.

When $k\ge 2$, in the limit as $\eps\to 0$, the effective obstacle lives on a subspace of codimension strictly larger than $1$. Here it is more natural to approximate $u$ with the solution to the unconstraint problem $$F(D^2h)=1 \text{ in $B_1$, and } h=u \text{ along $\partial B_1$}.$$ We show that $h$ `almost' solves the constrained problem, and its second order Taylor expansion gives the next approximation of $u$ in $B_\rho$ with an error of the order $(\eps-\eps^\mu)\rho^2$ for some $\mu>1$. For $\eps$ small, this improvement is much slower than $\eps\to(1-\beta)\eps$. Consequently, we need a much more delicate argument to keep track of the change in the polynomials at each step, essentially saying that if the improvement of error is small, then the change in the polynomials is even smaller. 

Combining these two cases together, we get a rate of convergence to the blow-up profile, which allows us to establish the main result Theorem \ref{MainResult}.

To our knowledge, this is the first structural result for singular points in the \OP with nonlinear operators. We hope that the ideas and techniques developed here can be applied to other types of \FB problems.

This paper is structured as follows. In the next section we provide some preliminary material and introduce some notations. In Section 3 we establish the main new observations of this paper, the improvement of  monotonicity and convexity of the solution. With these we prove two lemmata concerning the iterative  scheme. In Section 4 we deal with the case when $k=1$ as in \eqref{OrderingOfCoefficients}. In Section 5 we deal with the case when $k\ge 2.$ In the last section, we combine these to prove the main result.

%%%%%%%%%%%%%%%%%%%%%%%%%%%%%%%%%%%%%%%%%%%%%%%%%%%%%%%%%%%%%%%%%%%%%%%%%%%%%%%%%%%%%%%%%%%%%%%%%%%%%%%%%%%%%%%%%%%%%%%%%%%%%%%%%%%%%%%%%%%%%%%%%%%%%%%%%%%%
\section{Preliminaries and notations}

This section is divided into three subsections. In the first subsection we discuss some regularity properties of convex elliptic operators. The main reference for these is Caffarelli-Cabr\'e \cite{CC}. In the next subsection we include some known results on the obstacle problem, mostly from Lee \cite{L}. In the last subsection we recall an expansion of solutions to the thin obstacle problem. 

\subsection{Fully nonlinear convex elliptic operators}
Let $\mathcal{S}_d$ denote the space of $d$-by-$d$ symmetric matrices. Our  assumptions on the operator $F:\mathcal{S}_d\to\R$ are: 
\begin{equation}\label{FirstAssumption}F(0)=0; \text{ }F\text{ is convex};\end{equation} 
\begin{equation}\label{SecondAssumption}F \text{ is $C^{1,\alpha_F}$ for some $\alpha_F\in(0,1)$ with $C^{1,\alpha_F}$ seminorm } [F]_{C^{1,\alpha_F}}\le C_F;  \end{equation}
there is a constant $1\le\Lambda<+\infty$ such that \begin{equation}\label{Ellipticity}\frac{1}{\Lambda}\|P\|\le F(M+P)-F(M)\le\Lambda\|P\|\end{equation} for all $M,P\in\mathcal{S}_d$ and $P\ge 0.$

We call a constant \textit{universal} if it depends only on the dimension $d$, the elliptic constant $\Lambda$ and $C_F$, $\alpha_F$.

For a $C^2$ function $\varphi$, define the \textit{linearized operator} $L_\varphi:\mathcal{S}_d\to\R$ by \begin{equation*}\label{LinearizedOperator}L_\varphi(M)=\sum_{ij} F_{ij}(D^2\varphi)M_{ij},\end{equation*}where $F_{ij}$ denotes the derivative of $F$ in the $(i,j)$-entry, and $D^2\varphi$ is the Hessian of $\varphi$. One consequence of convexity is \begin{equation}\label{CompareWithLinearizedEquation}
L_v(w-v)\le F(D^2w)-F(D^2v)\le L_w(w-v).
\end{equation} 
As a result, we have the following comparison principle:
\begin{prop}\label{ComparisonPrinciple}
Let $u$ be a solution to \eqref{OP}. Suppose the functions $\Phi$ and $\Psi$ satisfy the following equations.
\begin{equation*}
\begin{cases}
F(D^2\Phi)\ge 1 &\text{ in $\Omega\cap\{\Phi>0\}$,}\\ \Phi\le u &\text{ on $\partial\Omega$.} 
\end{cases}
\end{equation*} 

\begin{equation*}
\begin{cases}
F(D^2\Psi)\le 1 &\text{ in $\Omega$,}\\\Psi\ge0 &\text{ in $\Omega$,}\\ \Psi\ge u &\text{ on $\partial\Omega$.} 
\end{cases}
\end{equation*} Then $\Phi\le u\le\Psi$ in $\Omega.$
\end{prop} 

\begin{proof}
Define $U=\Omega\cap\{\Phi>0\}$. Then \eqref{CompareWithLinearizedEquation} implies that inside $U$, we have  $$L_{\Phi}(\Phi-u)\ge F(D^2\Phi)-F(D^2u)\ge 0.$$ Note that $\partial U\subset\partial\Omega\cup\{\Phi=0\}$, we have $\Phi-u\le 0$ on $\partial U$ since $u\ge 0$ in $\Omega$ and $\Phi\le u$ on $\partial\Omega.$ Maximum principle gives $$\Phi-u\le 0 \text{ in $U$.}$$ Again with $u\ge 0$, we have $\Phi-u\le 0$ in $\{\Phi\le 0\}=U^c$. Combining these we have $$\Phi\le u \text{ in $\Omega$}.$$ 

To see the comparison between $u$ and $\Psi$, we define $V=\Omega\cap\{u>0\}.$ Then  \eqref{CompareWithLinearizedEquation} implies that inside $V$, we have 
$$L_u(\Psi-u)\le F(D^2\Psi)-F(D^2u)\le 1-1=0.$$ With $\Psi\ge 0$ in $\Omega$ and $\Psi\ge u$ on $\partial\Omega$, we see that $\Psi-u\ge 0$ on $\partial V$. Maximum principle leads to $$\Psi-u\ge 0 \text{ in $V$.}$$In $V^c$, $u=0\le\Psi$, thus $$\Psi\ge u \text{ in $\Omega$.}$$
\end{proof}

One corner stone of the regularity theory of fully nonlinear elliptic operators is the Evans-Krylov estimate \cite{CC}:
\begin{thm}\label{EvansKrylov} 
Let $F:\mathcal{S}_d\to\R$ be a convex operator satisfying $F(0)=0$ and \eqref{Ellipticity}. If $v$ solves $$F(D^2v)=f \text{ in $B_1$},$$then there are universal constants $\alpha\in(0,1)$ and $0<C<+\infty$ such that $$\|v\|_{C^{2,\alpha}(B_{1/2})}\le C(\|v\|_{\mathcal{L}^\infty(B_1)}+\|f\|_{C^{\alpha}(B_1)}).$$
\end{thm} 

In particular, if $u$ solves \eqref{OP}, then in $\PosS$ we have enough regularity to differentiate the equation and use convexity of $F$ to get\begin{equation}\label{EquationForDerivatives}
L_u(D_eu)=0,\text{ } L_u(D_{ee}u)\le 0 \text{ in $\PosS.$}
\end{equation} Here $e\in\Sph$ is a unit vector. Here and in later parts of the paper, $D_e$ denotes the differentiation in the $e$-direction. $D_{ee}$ denotes the pure second order derivative in the $e$-direction. When differentiating along directions of a standard orthonormal basis of $\R^d$, we also write $D_i=D_{e_i}$ and $D_{ij}=D_{e_i}D_{e_j}$, where $e_i$ is the $i$th vector in the standard basis. 

A direct application of the previous theorem gives the following estimate:
\begin{prop}\label{EstimateForDifference}
Let $F$ satisfy the assumptions \eqref{FirstAssumption}, \eqref{SecondAssumption} and \eqref{Ellipticity}. 

If $v$ solves $$F(D^2v)=1 \text{ in $B_1$},$$ and $p$ is a quadratic polynomial with $F(D^2p)=1,$ then $$\|v-p\|_{C^{2,\alpha}(B_{1/2})}\le C\|v-p\|_{\mathcal{L}^\infty(B_1)}$$ for some universal $\alpha\in(0,1)$ and $0<C<+\infty.$

If $w$ also solves $$F(D^2w)=1 \text{ in $B_1$},$$ then for a universal constant $\alpha\in(0,1)$ we have $$\|v-w\|_{C^{2,\alpha}(B_{1/2})}\le C\|v-w\|_{\mathcal{L}^\infty(B_1)}$$ where $C$ further depends on $\|v\|_{\mathcal{L}^\infty(B_1)}$ and $\|w\|_{\mathcal{L}^\infty(B_1)}.$
\end{prop} 

\begin{proof}
For the first statement in the proposition, we directly apply the previous theorem to the operator $$G(M)=F(M+D^2p)-F(D^2p).$$ This satisfies all assumptions in Theorem \ref{EvansKrylov}. The difference $v-p$ solves $$G(D^2(v-p))=0 \text{ in $B_1$}.$$

For the second statement of the proposition, we first apply Theorem \ref{EvansKrylov} to $v$ and $w$, which gives $$\|v\|_{C^{2,\alpha}(B_{3/4})}\le C\|v\|_{\mathcal{L}^\infty(B_1)}$$ and $$\|w\|_{C^{2,\alpha}(B_{3/4})}\le C\|w\|_{\mathcal{L}^\infty(B_1)}.$$

In $B_1$ the difference $v-w$ solves the linear equation $$A_{ij}(x)D_{ij}(v-w)=0$$ with coefficients $$A_{ij}(x)=\int_{0}^1F_{ij}(tD^2v(x)+(1-t)D^2w(x))dt.$$By the previous estimate, this  is H\"older continuous. We apply the standard Schauder theory to get the desired estimate.
\end{proof} 

Next we give an estimate for solutions to linear equations with coefficients which are close to being constant in a large portion of the domain. This is relevant in our analysis since often the linearized operators considered are perturbations of the constant coefficient operator $L_p$. 

\begin{prop}\label{v-w}
Let $\Omega$ be a Lipschitz domain. Suppose $v$,$w$ are $C^{2}$ solutions to the uniformly elliptic linear equations
$$ a^{ij}(x)v_{ij}=0 \quad \mbox{and} \quad b^{ij} w_{ij}=0 \quad \mbox{in $\Omega$}, \quad v=w=\varphi \quad \mbox{on $\partial \Omega$,} $$
with $\varphi$ H\"older continuous, and with coefficients $a^{ij}(x)$ measurable, $b^{ij}$ constant.

 If
$$|a^{ij}(x)-b^{ij}| \le \eta \quad \mbox{in} \quad \Omega_\eta:=\{x \in \Omega| \quad dist(x,\partial \Omega) > \eta\},$$ 
then
$$\|v-w\|_{L^\infty(\Omega)} \le \omega(\eta), $$
where $\omega(\cdot)$ is a modulus of continuity, which depends on the ellipticity constants, the domain $\Omega$ and $\|\varphi\|_{C^\alpha}$.
\end{prop}

\begin{proof}
The proposition follows from the perturbative methods developed in [CC]. Here we only sketch a proof by compactness. 

The global version of Harnack inequality implies that $v$, $w$ are uniformly H\"older continuous in $\overline \Omega$. Now we consider a sequence of $\eta_k \to 0$, and the 
corresponding solutions $v_k$, $w_k$ (for equations with coefficients $a^{ij}_k(x)$, $b^{ij}_k$) . Then, up to subsequences, they must converge uniformly to a solution of the same constant coefficient equation. The limiting solutions must coincide since they have the same boundary data, and the conclusion follows.
\end{proof}

%%%%%%%%%%%%%%%%%%%%%%%%%%%%%%%%%%%%%%%%%%%%%%%%%%%%%%%%%%%%%%%%%%%%%%%%%%%%%%%%%%%%%%%%%%%%%%%%%%%%%%%%%%%%%%%%%%%%%%%%%%%%%%%%%%%%%%%%%%%%%%%%%%%%%%%%%%%%
\subsection{Known results for the obstacle problem}In this subsection we include some classical results concerning the obstacle problem \eqref{OP}. Most of the results here can be found in Lee \cite{L}.

We begin with the optimal regularity of the solution:

\begin{prop}\label{C11}
Let $u$ be a solution to \eqref{OP}. Then for a compact set $K\subset\Omega$, $$\|u\|_{C^{1,1}(K)}\le C$$ for some $C$ depending on universal constants, $K$, and $\|u\|_{\mathcal{L}^\infty(\Omega)}.$
\end{prop} 

A direct consequence  is that in the contact set $\{u=0\}$,  we have\begin{equation}\label{Contact}
\nabla u=0 \text{ and } D^2u\ge 0 \text{ in the viscosity sense.}
\end{equation}

We have the following almost convexity estimate
\begin{prop}\label{UniformConvexity}
Let $u$ be a solution to \eqref{OP} in $\Omega=B_1$ with $u(0)=0.$ Then for some universal constants $\delta_0>0$ and $C$, $$D^2u(x)\ge -C|\log|x||^{-\delta_0} \text{ in $B_{1/2}.$}$$\end{prop} 

The \FB decomposes into the regular part and the singular part $$\partial\PosS=Reg(u)\cup\Singu,$$ with $Reg(u)$ given locally by a 
$C^{1,\alpha}$ surface which separates the $0$ set from the positivity set. 

Define the \textit{thickness} function of a set $E$, $\delta_E(\cdot)$, as $$\delta_E(r)=\frac{MD(E\cap B_r)}{r},$$where $MD(E\cap B_r)$ is the infimum of distances between two pairs of parallel hyperplanes such that $E\cap B_r$ is contained in the strip between them. 

Geometrically the singular set $\Singu$ is characterized by the vanishing thickness of the zero set:
\begin{prop}\label{VanishingThickness}
Let $u$ be a solution to \eqref{OP} in $B_1$ with $0\in\Singu.$ There is a universal modulus of continuity $\sigma_1$ such that $$\delta_{\{u=0\}}(r)\le\sigma_1(r).$$
\end{prop} In particular, if $0\in\Singu$, the zero set $\{u=0\}$ cannot contain a nontrivial cone with vertex at $0$.

Another characterization of the singular set  is that at points in $\Singu$ the solution  is approximated by quadratic polynomials. 

For this, we define the following class of polynomial solutions to the obstacle problem.  We also define the class of \textit{ convex} polynomials that do not necessarily satisfy the non-negative constraint.

\begin{defi}\label{QuadraticSolution}The class of \textit{quadratic solutions} is defined as $$\Qua=\{p:p(x)=\frac{1}{2}x\cdot Ax, A\ge 0, F(A)=1\}.$$

The class of \textit{unconstraint convex quadratic solutions} is defined as $$\UQua=\{p:p(x)=\frac{1}{2}x\cdot Ax+b\cdot x,A\ge 0, F(A)=1\}.$$
\end{defi} Here and in later parts of the paper, $x\cdot y$ denotes the standard inner product between two vectors $x$ and $y$.

Note that for a polynomial $p\in\UQua$, $D^2p\ge 0$. Ellipticity  \eqref{Ellipticity} then gives $$D^2p\le C \, I,$$ for some universal $C$.

For points in $\Singu$, we have the following uniform approximation by quadratic solutions:

\begin{prop}\label{UniformApproximation}
Let $u$ be a solution to \eqref{OP} in $B_1$ with $0\in\Singu$. There is a universal modulus of continuity $\sigma_2$ such that for each $r\in(0,1/2)$, there is $p^r\in\Qua$ satisfying $$\|u-p^r\|_{\mathcal{L}^\infty(B_r)}\le\sigma_2(r)r^2.$$
\end{prop} 

Combining Proposition \ref{UniformConvexity} and Proposition \ref{UniformApproximation}, we know that after some rescaling, our solution is in the following class:

\begin{defi}\label{ApproxClass}
Given $\eps,r\in(0,1)$ and $p\in\UQua$, we say that $u$ is \textit{$\eps$-approximated by the polynomial $p$ in $B_r$}, and use the notation $$u\in\mathcal{S}(p,\eps,r)$$
  if $$u \text{ solves \eqref{OP} in $B_r$,}$$ $$|u-p|\le\eps r^2 \text{ in $B_r$},$$ and 
  \begin{equation}\label{convexity}
  D^2u\ge-c_0\eps \,  I\text{ in $B_r,$}
  \end{equation} where $c_0=\frac{1}{16\Lambda^2}$.

\end{defi} 

The universal bound $0\le D^2p\le C \, I$ for $p\in\UQua$ immediately gives a universal bound on the size of $u$ whenever $u\in\mathcal{S}(p,\eps,r):$ \begin{equation}\label{UniversalBound}
0\le u\le C \text{ in $B_r$}
\end{equation}  where $C$ is universal.

%%%%%%%%%%%%%%%%%%%%%%%%%%%%%%%%%%%%%%%%%%%%%%%%%%%%%%%%%%%%%%%%%%%%%%%%%%%%%%%%%%%%%%%%%%%%%%%%%%%%%%%%%%%%%%%%%%%%%%%%%%%%%%%%%%%%%%%%%%%%%%%%%%%%%%%%%%%%
\subsection{The thin obstacle problem}
In this subsection we discuss solutions to the thin obstacle problem. In certain cases, our solution converges to them after normalization. Readers interested in the thin obstacle problem may consult Athanasopoulos-Caffarelli-Salsa \cite{ACS} or Petrosyan-Shahgholian-Uraltseva \cite{PSU}. In its most basic form, the thin obstacle problem is the following system:
\begin{equation}\label{ThinObstacle}
\begin{cases}
\Delta v\le 0 &\text{ in $B_1$,}\\ \Delta v=0 &\text{ in $B_1\cap(\{v>0\}\cup\{x_1\neq 0\})$,}\\ v\ge 0 &\text{ along $\{x_1=0\}$}.
\end{cases}
\end{equation} 
Here $x_1$ denotes the first coordinate function of $\R^d$.

For solutions to this problem, we have the following effective expansion according to frequencies at $0$:

\begin{prop}\label{ExpansionForThinObstacle}
Let $v$ be a non-trivial solution to \eqref{ThinObstacle} with $v(0)=0$. 

Then one of the following three possibilities happens for $v$:

\begin{enumerate}
\item{For some $a_\pm\in\R$ not both $0$, $$v(x)=a_+x_1^++a_-x_1^-+o(|x|)$$ as $x\to0;$}
\item{For some $r>0$ and $e\in\Sph\cap\{x_1=0\}$, $$D_ev>0 \text{ in $B_r\cap\{x_1\neq0\}$};$$}
\item{For some $A\in\mathcal{S}_d$ with  $e\cdot Ae\ge 0$ $\forall e\in\Sph\cap\{x_1=0\}$ and $trace(A)=0$, $$v(x)=\frac{1}{2}x\cdot Ax+o(|x|^2)$$ as $x\to 0$.}
\end{enumerate}
\end{prop} 

 For a real number $x$, $x^+$ and $x^-$ denote the positive and negative parts of $x$ respectively. Recall that $D_e$ denotes the differentiation in the $e$-direction.  

\begin{proof}
The Almgren frequency of $v$ at $0$ is well-defined. Denote this frequency by $\varphi$, then there are three possibilities: $\varphi=1$; or  $\varphi=3/2$; or  $\varphi\ge 2.$

If $\varphi=1$, then $v$ blows up to a $1$-homogeneous solution to \eqref{ThinObstacle}. In this case, possibility (1) as in the statement of the lemma holds. 

Similarly, if $\varphi\ge 2$, then possibility (3) happens.

When $\phi=3/2$, then $v$ blows up to a $3/2$-homogeneous solution. In this case $v$ is monotone in a direction in the hyper-plane $\{x_1=0\}$. This corresponds to possibility (2).

For details, the reader may consult \cite{ACS} or \cite{PSU}.
\end{proof}

%%%%%%%%%%%%%%%%%%%%%%%%%%%%%%%%%%%%%%%%%%%%%%%%%%%%%%%%%%%%%%%%%%%%%%%%%%%%%%%%%%%%%%%%%%%%%%%%%%%%%%%%%%%%%%%%%%%%%%%%%%%%%%%%%%%%%%%%%%%%%%%%%%%%%%%%%%%%
\section{Improvement of monotonicity and convexity}

In this section are some new observations concerning the directional monotonicity and convexity of solutions to the obstacle problem. They are at the heart of the further development of the theory.  

Roughly speaking, if the solution is `almost' monotone/ convex in $B_1$ and strictly monotone/ convex away from the free boundary, then the results here imply that the solution is indeed monotone/ convex in $B_{1/2}$. As already evident in the classical work of Caffarelli \cite{C1},  it is of fundamental importance  to develop such tools to transfer information away from the \FB to the full domain. 

Before we state the main results of this section, we begin with the construction of a barrier function. In the following lemma, $\gamma$ is the constant such that \begin{equation}\label{Gam}F(\gamma I)=1.\end{equation}Here $I$ is the identity matrix.  By \eqref{Ellipticity}, $\frac{1}{\Lambda}\le\gamma\le\Lambda.$

\begin{lem}\label{BarrierLemma}
For $0<\eta<r<1$ and $N>8\gamma \, r^2$, let $w$ be the solution to the following system
$$\begin{cases}
F(D^2w)=1 &\text{ in $B_r$,}\\w=\frac{1}{2}\gamma|x|^2 &\text{ along $\partial B_r\cap\{|x_1|>\eta\}$,}\\w=N &\text{ along $\partial B_r\cap\{|x_1|\le\eta\}$.}
\end{cases}$$

For $x_0\in B_{r/2}$, define $$w_{x_0}(x)=w(x)-w(x_0)-\nabla w(x_0)\cdot(x-x_0).$$ 

There is $\bar{\eta}$ depending on $r$, $N$ and universal constants, such that if $\eta<\bar{\eta}$, then for all $x_0\in B_{r/2}$, $w_{x_0}$ satisfies $$w_{x_0}(x)\ge \frac{1}{64}\gamma|x-x_0|^2 \text{ in $B_r$,}$$and $$w_{x_0}\ge\frac{1}{2}N\text{ along $\partial B_r\cap\{|x_1|\le\eta\}$}.$$
\end{lem} 

\begin{proof}
Define $\varphi=w-\frac{1}{2}\gamma|x|^2$ in $B_r$. Then Proposition \ref{EstimateForDifference} gives $$\|\varphi\|_{C^{2,\alpha}(B_{\frac{3}{4}r})}\le C_r\|\varphi\|_{\mathcal{L}^\infty(B_{\frac{7}{8}r})},$$ for some $C_r$ depending on universal constants and $r$.

We claim that as $\eta\to0$, 
\begin{equation}\label{convto0}
\mbox{$\varphi$ converges locally uniformly in $B_r$ to $0$. }
\end{equation}
Consequently, there is a modulus of continuity $\omega$, depending on universal constants, $N$ and $r$, such that $$C_r\|\varphi\|_{\mathcal{L}^\infty(B_{\frac{7}{8}r})}\le\omega(\bar{\eta})$$ whenever $\eta<\bar{\eta}.$ Thus the previous estimate gives \begin{equation}\label{Barrier1}
\|\varphi\|_{C^{2,\alpha}(B_{\frac{3}{4}r})}\le \omega(\bar{\eta})
\end{equation}  whenever $\eta<\bar{\eta}.$

In order to prove the claim \eqref{convto0} we notice that $\varphi$ satisfies a linear elliptic equation 
$$a_{ij}(x) \varphi_{ij}=0 \quad \mbox{in} \quad B_r,$$
with ellipticity constant $\Lambda$. Also, $\varphi$ vanishes on $\partial B_r$ except on $\partial B_r\cap\{|x_1|\le\eta\}$ where $N \ge \varphi \ge \frac 34 N$. We extend $\varphi=0$ outside $B_r$, and by the weak Harnack inequality it follows that $\max \, \varphi$ decreases geometrically on the outward dyadic regions centered around a point $y \in \partial B_r \cap \{x_1=0\}$,
$$ B_{2^{1-k}}(y) \setminus B_{2^{-k}}(y) \quad \quad \mbox{as long as} \quad  \eta \le 2^{-k} \le \frac r 4.$$
We easily obtain the claim \eqref{convto0} as we let $\eta \to 0$.

Define $\varphi_{x_0}(x)=\varphi(x)-\varphi(x_0)-\nabla\varphi(x_0)\cdot(x-x_0)$. 
The conclusion follows from \eqref{Barrier1} by using 
$$w_{x_0}=\varphi_{x_0} + \frac{1}{2}\gamma|x-x_0|^2,$$
and $\varphi \ge 0$ in $B_r$, $\varphi \ge \frac 34 N$ on $\partial B_r\cap\{|x_1|\le\eta\}$.
\end{proof} 

With this we prove the following improvement of monotonicity lemma. Recall the class of solutions $\mathcal{S}(p,\eps,r)$ is defined in Definition \ref{ApproxClass}, and that $D_e$ is the differentiation along direction $e$.

\begin{lem}\label{Monotonicity1}
Suppose $u\in\mathcal{S}(p,\eps,r)$  satisfies the following for some constants $K$, $\sigma$, and $0<\eta<r$,  and a direction $e\in\Sph$:

$$D_eu\ge -K\eps \text{ in $B_r$},$$ and $$D_eu\ge\sigma\eps \text{ in $B_r\cap\{|x_1|\ge\eta\}$}.$$

There is $\bar{\eta}$, depending on universal constants, $r$, $\sigma$ and $K$, such that if $\eta\le\bar{\eta}$, then $$D_eu\ge 0 \text{ in $B_{r/2}$.}$$
\end{lem} 

\begin{proof}
Choose $c>0$ small, depending on universal constants and $\sigma$, such that $$c\|u\|_{\mathcal{L}^\infty(B_r)}\le\sigma.$$ Then define $N=\max\{4K/c,10\|u\|_{\mathcal{L}^\infty(B_r)}\},$ depending only on $K,\sigma$ and universal constants since we have the universal bound \eqref{UniversalBound}.

Let $\bar{\eta}$ be the constant given in Lemma \ref{BarrierLemma}, depending on $N$ and $r$. Let $w_{x_0}$ be the barrier as in that lemma. Assume $\eta<\bar{\eta}.$

If we define $U=B_r\cap\PosS$ and pick $x_0\in B_{r/2}\cap\PosS$, then on $\partial U\subset(\partial B_r\cap\{|x_1|\ge\eta\})\cup(\partial B_r\cap\{|x_1|<\eta\})\cup\partial\PosS$, one has 

\begin{equation*}
c\eps(u-w_{x_0})(x)\le c\eps u\le\sigma\eps \text{ along $\partial B_r\cap\{|x_1|\ge\eta\}$;}
\end{equation*}

\begin{align*}
c\eps(u-w_{x_0})(x)&\le c\eps(u-\frac{1}{2}N)\\&\le c\eps(-\frac{1}{4}N)\\&\le-K\eps \text{ along $\partial B_r\cap\{|x_1|\le\eta\}$;}
\end{align*}
and $$c\eps(u-w_{x_0})(x)\le 0 \text{ along $\partial\PosS$.}$$

Our assumptions on $D_eu$ and \eqref{Contact} imply $$D_eu\ge c\eps(u-w_{x_0})\text{ along $\partial U$.}$$

Now with  \eqref{CompareWithLinearizedEquation} and \eqref{EquationForDerivatives}, we have $L_u(D_eu)=0$ and $L_u(u-w_{x_0})\ge 0$ in $U$. Thus comparison principle gives $$D_eu(x_0)\ge c\eps(u-w_{x_0})(x_0)=c\eps u(x_0)\ge 0.$$

Since this is true for all $x_0\in B_{r/2}\cap\PosS$ and $D_eu=0$ in $\{u=0\}$, $D_eu\ge 0$ in $B_{r/2}$.
\end{proof} 

A slightly different version is also useful:

\begin{lem}\label{Monotonicity2}

Suppose $u\in\mathcal{S}(p,\eps,r)$  satisfies the following for some constants $K$, $\sigma$, and $0<\eta<r$,  and a direction $e\in\Sph$:

$$D_eu\ge -K\eps \text{ in $B_r$},$$ $$D_eu\ge0 \text{ in $B_r\cap\{|x_1|\ge\eta\}$},$$ and 
$$D_eu\ge\sigma\eps \text{ in $B_r\cap\{u>\frac{1}{256}\gamma r^2\}$},$$

There is $\bar{\eta}$, depending on universal constants, $r$, $\sigma$ and $K$, such that if $\eta\le\bar{\eta}$, then $$D_eu\ge 0 \text{ in $B_{r/2}$.}$$
\end{lem} 

\begin{proof}
The proof is almost the same as the previous proof. The only difference happens for the comparison along the boundary $\partial B_r\cap\{|x_1|\ge\eta\}.$

On $\partial B_r\cap\{|x_1|\ge\eta\}\cap\{u>\frac{1}{256}\gamma r^2\},$ we still have $D_eu\ge\sigma \eps$, and the same comparison $D_eu\ge c\eps(u-w_{x_0})$ holds. 

On $\partial B_r\cap\{|x_1|\ge\eta\}\cap\{u\le\frac{1}{256}\gamma r^2\},$ we invoke $$u-w_{x_0}\le u-\frac{1}{64}\gamma|x-x_0|^2\le u-\frac{1}{64}\gamma (1/2r)^2\le 0$$ for $x_0\in B_{r/2}.$ Thus along this piece of the boundary we still have $$D_eu\ge 0\ge c\eps(u-w_{x_0}).$$
\end{proof}

Finally we have the following improvement of convexity estimate:
\begin{lem}\label{Convexity}
Suppose $u\in\SPE.$ There is a universal constant $C$ such that if $\Dee p\ge C\eps$ along some direction $e\in\Sph$, then $$\Dee u\ge 0 \text{ in $B_{1/2}$.}$$
\end{lem} 

\begin{proof}
Let $\gamma$ be the constant as in \eqref{Gam}, and $c_0$ be the constant as in Definition \ref{ApproxClass}. 

For $x_0\in B_{1/2}\cap\PosS$,  define $$h(x)=\Dee u(x)-\frac{64c_0}{\gamma}\eps(u(x)-\frac{1}{2}\gamma |x-x_0|^2).$$ Define $U=B_{3/4}\cap\PosS$, we have the following

\textit{Claim: For some universal constant $C$, if $\Dee p\ge C\eps$, then $h\ge 0$ along $\partial U$.}

Note that by  \eqref{CompareWithLinearizedEquation} and \eqref{EquationForDerivatives}, $$L_u(h)\le 0  \text{ in $U$.}$$ Thus once the claim is proved,  $h\ge 0$ in $U$ by maximum principle. In particular, we have $\Dee u(x_0)\ge 0$. Together with \eqref{Contact}, we have $\Dee u\ge 0$ in the entire $B_{1/2}$.

Therefore, it suffices to prove the claim. 

First we note that along $\partial\PosS$, $\Dee u\ge 0$ and $u=0$, thus $h\ge 0$ along this part of $\partial U$. 

We divide the other part $\partial B_{3/4}\cap\{u>0\}$ into two pieces $$\partial B_{3/4}\cap\{u>0\}=(\partial B_{3/4}\cap\{u\le \frac{1}{64}\gamma\})\cup (\partial B_{3/4}\cap\{u> \frac{1}{64}\gamma\}).$$

Along the first piece $\partial B_{3/4}\cap\{u\le \frac{1}{64}\gamma\}$, 
\begin{align*}
h(x)&\ge -c_0\eps-\frac{64c_0}{\gamma}\eps(u(x)-\frac{1}{2}\gamma |x-x_0|^2)\\&\ge-c_0\eps-\frac{64c_0}{\gamma}\eps(\frac{1}{64}\gamma-\frac{1}{32}\gamma)\\&=0.
\end{align*}

It remains to deal with $y_0\in  \partial B_{3/4}\cap\{u> \frac{1}{64}\gamma\}.$

Firstly the universal bound \eqref{UniversalBound} and Proposition \ref{C11} give a universal $r_0>0$ such that $$dist(y_0,\{u=0\})\ge r_0.$$

In particular $F(D^2u)=1$ in $B_{r_0}$, and we can apply Proposition \ref{EstimateForDifference} to get $$|D^2u(y_0)-D^2p|\le C_0\eps$$ for a universal constant $C_0.$ Therefore, $\Dee u(y_0)\ge C\eps-C_0\eps.$ Consequently, for $y_0\in  \partial B_{3/4}\cap\{u> \frac{1}{64}\gamma\},$
\begin{align*}
h(y_0)&\ge (C-C_0)\eps-\frac{64c_0}{\gamma}\eps u(y_0)\\&\ge\eps(C-C_0-\frac{64c_0}{\gamma}\max_{B_1}u).
\end{align*}Again note the universal bound on $\max u$ as in \eqref{UniversalBound}, if we choose $C$ universally large, then $h\ge 0$ on this last piece of $\partial U$. 

This completes the proof for the claim.
\end{proof}

%%%%%%%%%%%%%%%%%%%%%%%%%%%%%%%%%%%%%%%%%%%%%%%%%%%%%%%%%%%%%%%%%%%%%%%%%%%%%%%%%%%%%%%%%%%%%%%%%%%%%%%%%%%%%%%%%%%%%%%%%%%%%%%%%%%%%%%%%%%%%%%%%%%%%%%%%%%%
\section{Quadratic approximation of solution: Case 1}

In this section and the next, we use the technical tools developed in the previous sections to study the behaviour of our solution near a singular point, say, $0\in\Singu.$ 

The classical approach is to study the rescales of $u$, $$u_r(x)=\frac{1}{r^2}u(rx)$$ as $r\to 0$. Proposition \ref{C11} gives enough compactness to get convergence of $u_{r_j}$ to some quadratic polynomial, say $p$,  along  a subsequence $r_j\to 0.$ If the limit does not depend on the particular subsequence, then there is a well-defined stratification of $\Singu$ depending on the dimension of $ker(D^2p)$. If there is a rate of convergence of $u_r\to p$, then we get regularity of the singular set near $0\in\Singu.$

With the help of monotonicity formulae, this program has been executed with various degrees of success in  \cite{C2}, \cite{CSV}, \cite{FSe}, \cite{M} and \cite{W}. One idea behind these works is that once $u_{r_0}$ is close $p$ for a particular $r_0$, then monotonicity formulae imply $u_r$ remains close to $p$ for all $r<r_0$. 

Since no monotonicity formula is available in our problem, we do not have access to all small scales. Instead, we proceed by performing an iterative scheme. Let $\rho\in(0,1)$, the building block of this scheme is to study the following question: If $u$ is close to $p$ in $B_1$, can we approximate $u$ better in $B_{\rho}$?

Quantitatively, we seek to prove the following: 

If $|u-p|<\eps$ in $B_1$ for some small $\eps$, then we can find a quadratic polynomial $q$ such that $|u-q|<\eps'\rho^2$ in $B_\rho$, where $\eps'<\eps$.

The rate of decay  $\eps\to\eps'$ is linked to the rate of convergence in the blow-up procedure. 

Define the normalized solution $\hat{u}_\eps=\frac{1}{\eps}(u-p),$ and suppose we can show that as $\eps\to 0$, $\hat{u}_\eps\to \hat{u}_0.$  Then the formal expansion $$u=p+\eps\hat{u}_0+\eps o(1)$$ shows that a better approximation in $B_\rho$ follows if $\hat{u}_0$ is $C^2$ near $0$. 

To this end, we need to separate two different cases.  

Let $\lambda_1\ge\lambda_2\ge\dots\ge\lambda_d\ge 0$ denote the eigenvalues of $D^2p$. Depending on their sizes, the contact set $\{u=0\}$ concentrates along subspaces of various dimensions.  When $\lambda_2\le C\eps$, $p\sim \frac{1}{2}(x\cdot e)^2$ and the contact set concentrates along a $(d-1)$-dimensional subspace $\{x\cdot e=0\}.$  When $\lambda_2\gg\eps$, the contact set concentrates along a subspace with higher co-dimension. 

In this section, we deal with the first case when $\lambda_2\le C\eps$. 

In this case $\hat{u}_\eps$ converges to the solution of the thin obstacle problem \eqref{ThinObstacle}, and in particular Proposition \ref{ExpansionForThinObstacle} applies to the limit $\hat{u}_0.$ To show $\hat{u}_0$ is $C^2$ near $0$, we need to rule out possibilities (1) and (2) as in the statement of Proposition \ref{ExpansionForThinObstacle}. This can be achieved using explicit barriers and the lemmata in the previous section. 

In this section, we decompose $\R^d=\R\times\R^{d-1}$ and write $x=(x_1,x')$, where $x'$ is the projection of $x$ onto the subspace $\{x_1=0\}.$ Similarly, for $E\subset\R^d$, we define $$E'=E\cap\{x_1=0\}.$$

The \textit{main result} of this section is the following:

\begin{lem}[Quadratic approximation of solution: Case 1]\label{QuadraticApproxTopStratum}
Suppose for some $\kappa>0$, we have $$u\in\SPE$$ for some $p\in\Qua$ with $$\lambda_2(D^2p)\le\kappa\eps,$$ and $$0\in\Singu.$$

There are constants $\bar{\eps},\beta\in(0,1)$ and $\bar{r}\in(0,1/2)$, depending on universal constants and $\kappa$, such that if $\eps<\bar{\eps}$, then $$u\in\mathcal{S}(p',\eps',r)$$ for some $p'\in\Qua$, $\eps'=(1-\beta)\eps$ and $r\in(\bar{r},1/2).$
\end{lem} 

The class of quadratic solutions $\mathcal{Q}$ and the class of well-approximated solution $\SPE$ are defined in Definition \ref{QuadraticSolution} and Definition \ref{ApproxClass}.  

Here and in later parts of the paper, $\lambda_j(M)$ denotes the $j$th largest eigenvalue of the matrix $M$. 

\begin{rem}\label{Remark}The parameter $\kappa$ will be chosen in the final section, depending only on universal constants. After that, all constants in this lemma become universal. \end{rem}

We begin with some preparatory lemmata.

\begin{lem}\label{UniformLipschitz}
Let $u$ and $p$ be as in Lemma \ref{QuadraticApproxTopStratum}.

Then $$|\nabla(u-p)|\le L\eps \text{ in $B_{1/2}$,}$$ where $L$ depends only on universal constants and $\kappa$. 
\end{lem}

\begin{proof}Define the normalization $$\hat{u}=\frac{1}{\eps}(u-p).$$
Since $F(D^2u)\le 1$ and $F(D^2p)=1$ in $B_1$, by \eqref{CompareWithLinearizedEquation} we have 
 \begin{equation}\label{SuperSolutionToLaplace0}L_p \hat u= \frac 1 \eps L_p(u-p)\le 0 \text{ in $B_1.$}
\end{equation}

Up to a rotation, the polynomial  $p$ is of the form $$p(x)=\frac{1}{2}\sum a_jx_j^2$$ with $a_1\ge a_2\ge\dots\ge a_d\ge 0$ and $a_2\le C\kappa\eps$ for some universal constant $C$. Then $D^2u\ge-c_0\eps \text{ in $B_1$,}$ and $\Dee p\le C\kappa\eps$ for all $e\in\Sph\cap\{x_1=0\},$ gives $$\Dee\hu\ge -C \text{ in $B_1$ for all $e\in\Sph\cap\{x_1=0\}$}$$ for some $C$ depending only on universal constants and $\kappa$. Now the result easily follows from this and the fact that $\hat u \in C^{1,1}$ satisfies \eqref{SuperSolutionToLaplace0}.

Indeed, after a linear deformation, we can assume that $L_p=\triangle$ and the inequality on $D_{ee} \hat u$ is still satisfied after relabeling the constant $C$. Then $\Delta\hat{u}\le 0$ implies that we also have $$D_{11}\hu\le C \text{ in $B_1$.}$$

Together with $|\hu|\le 1$ in $B_1$, these imply $$|\nabla\hu|\le C \text{ in $B_{1/2},$}$$ for some $C$ depending only on universal constants and $\kappa.$ 
\end{proof}  

This lemma provides us with enough compactness for a family of normalized solutions. Actually it even allows us to consider a family of nomalized solutions to the obstacle problem involving a family of different operators. This is necessary to get uniform estimates.

To fix ideas, let $\{F_j\}$ be a sequence of operators satisfying the same assumptions that we have on our operator $F$, namely, \eqref{FirstAssumption}, \eqref{SecondAssumption} and \eqref{Ellipticity}. 

For each $F_j$, there is a unique $\gamma_j$ such that $$F_j(\gamma_je_1\otimes e_1)=1.$$ Ellipticity implies $\gamma_j\in[1/\Lambda,\Lambda].$ Define the associated polynomial \begin{equation}\label{Poly}q_j(x)=\frac{1}{2}\gamma_jx_1^2.\end{equation}

Then we have the following lemma, that identifies the problem solved by the limit of nomalized solutions:

\begin{lem}\label{LimitingProblem}
Let $F_j$ be a sequence of operators satisfying the same assumptions as in \eqref{FirstAssumption}, \eqref{SecondAssumption} and \eqref{Ellipticity}. Let $u_j$ solve the obstacle problem \eqref{OP} with operator $F_j$ in $B_1.$

Suppose for some constant $\kappa>0$ and a sequence $\eps_j\to 0$,  there are polynomials $$p_j(x)=\frac{1}{2}\sum a^j_ix_i^2$$ with $a^j_1\ge a^j_2\ge\dots a^j_d\ge 0,$ $a^j_2\le\kappa\eps_j,$ $$F_j(\sum a^j_ie_i\otimes e_i)=1,$$  and $$|u_j-p_j|\le\eps_j \text{ in $B_1$.}$$

Then up to a subsequence, the normalized solution $$\hat{u}_j=\frac{1}{\eps_j}(u_j-q_j)$$ converges locally uniformly in $B_1$ to some $\hat{u}_\infty$, where $q_j$ is the polynomial as in \eqref{Poly}. 

Moreover, up to scaling, $\hat{u}_\infty$ solves the thin obstacle problem as in \eqref{ThinObstacle}.\end{lem} 

\begin{proof}
Lemma \ref{UniformLipschitz} gives locally uniform $C^{0,1}$ bound on the family $\{\hat{u}_j\}$. Consequently, up to a subsequence they converge to some $\hat{u}_\infty$ locally uniformly in $B_1$. 

Define the operator $G_j$ by $$G_j(M)=\frac{1}{\eps_j}(F_j(\eps_jM+D^2q_j)-F_j(D^2q_j)).$$ Then $$G_j(D^2\hat{u}_j)=\frac{1}{\eps_j}(\chi_{\{u_j>0\}}-1)=-\frac{1}{\eps_j}\chi_{\{u_j=0\}}.$$

By uniform $C^{1,\alpha_F}$ estimate on the family $\{F_j\}$,  up to a subsequence $G_j$ locally uniformly converges to some linear elliptic operator. Up to a scaling, we assume this limiting operator is the Laplacian. 

Then $G_j(D^2\hat{u}_j)\le 0$ for all $j$ implies $$\Delta\hat{u}_\infty\le 0 \text{ in $B_1$.}$$

If $x_0\in\{\hat{u}_\infty>0\}$, then $\hat{u}_j>0$ in a neighborhood of $x_0$ for large $j$. Note that $q_j\ge0$, thus $u_j>0$ in a neighborhood of $x_0$ for large $j$. Thus $G_j(D^2\hat{u}_j)=0$ in a neighborhood of $x_0$ for all large $j$. Consequently, $\Delta \hat{u}_\infty(x_0)=0.$ That is, 
$$\Delta \hat{u}_\infty=0 \text{ in $\{\hat{u}_\infty>0\}$.}$$

Meanwhile, for $x\in\{x_1\neq 0\}$, $u_j(x)\ge p_j(x)-\eps_j\ge c|x_1|^2-\eps_j$, where $c$ is a universal positive constant. Thus $u_j>0$ in a neighborhood of $x$ for large $j$. Consequently $G_j(D^2\hat{u}_j)=0$ in a neighborhood of $x$ for large $j$. Thus $\Delta\hat{u}_\infty(x)=0.$ That is, $$\Delta\hat{u}_\infty=0 \text{ in $\{x_1\neq 0\}$.}$$

It remains to show that $\hat{u}_\infty\ge 0$ along $\{x_1=0\}$. For this, simply note that $u_j\ge 0$ and $q_j=0$ for all $j$ along $\{x_1=0\}$. \end{proof}

Now we start the proof of Lemma \ref{QuadraticApproxTopStratum}. As explained  at the beginning of this section, the normalized solutions converge to a solution to the thin obstacle problem. The key to the improvement in approximation is the show this limit is $C^2$ at $0$, that is, possibilities (1) and (2) as in Proposition \ref{ExpansionForThinObstacle} cannot happen. 

\begin{proof}[Proof of Lemma \ref{QuadraticApproxTopStratum}]

Let $\bar{r}, \beta\in(0,1)$ be small constants to be chosen, depending only on universal constants and $\kappa$.

Suppose there is no $\bar{\eps}>0$ satisfying the statement of the lemma. For a sequence of $\eps_j\to 0,$ a sequence of operators $F_j$ satisfying the assumptions \eqref{FirstAssumption}, \eqref{SecondAssumption} and \eqref{Ellipticity}, we have a sequence of solutions to \eqref{OP} with these operators such that $$u_j\in\mathcal{S}(p_j,\eps_j,1)$$ for some $p_j\in\Qua$ with $\lambda_2(D^2p_j)\le\kappa\eps_j,$ and $$0\in\Sigma(u_j),$$  but $$u_j\not\in\mathcal{S}(q,(1-\beta)\eps_j,r)$$ for any $q\in\Qua$ and $r\in(\bar{r},1/2).$

Up to a rotation, we assume 

 $$p_j(x)=\frac{1}{2}\sum a^j_ix_i^2$$ with $a^j_1\ge a^j_2\ge\dots a^j_d\ge 0, \text{ and }a^j_2\le\kappa\eps_j.$

 Define $$\hat{u}_j=\frac{1}{\eps_j}(u_j-q_j),$$ where $q_j(x)=\frac{1}{2}\gamma_jx_1^2$ with $F_j(\gamma_je_1\otimes e_1)=1.$ Then up to a scaling, Lemma \ref{LimitingProblem} shows that up to a subsequence, $$\hat{u}_j\to\hat{u}  \text{ locally uniformly in $B_1$,}$$ where $\hat{u}$ solves the thin obstacle problem \eqref{ThinObstacle}.
 
 Moreover, $u_j(0)=0$ for all $j$ implies $\hat{u}(0)=0.$ Lemma \ref{UniformLipschitz} gives a $C_{loc}^{0,1}(B_1)$ bound on $\hat{u}.$  Consequently, Proposition \ref{ExpansionForThinObstacle} is applicable for $\hat{u}$.  
 
We show that possibilities (1) and (2) of Proposition \ref{ExpansionForThinObstacle} cannot happen for $\hat{u}$. 

\

\textit{Step 1: Possibility (1) as in Proposition \ref{ExpansionForThinObstacle} does not happen for $\hat u$.}

Suppose it happens, then we have 
$$\hat{u}=a_+x_1^++a_-x_1^ +  + o(|x|)$$ as $x\to 0.$
First we show that $a_\pm \le 0$. 

Assume that $$a_+ >0,$$ and then we use a barrier to show that $u(0)>0$, contradicting $0\in \Sigma(u).$

For this we choose $r$ small such that 
\begin{equation}\label{hat1}
\hat{u} > \frac 12 d \, \Lambda^2 x_1^2 - \frac 12 |x'|^2,
\end{equation}
near $\partial U_r \cap \{x_1 \ge 0\}$ where $U_r$ is the cylinder of size $r$, 
$$U_r:=B_r' \times [-r,r].$$
This means that $\hat u_j$ satisfies the same inequality \eqref{hat1} above for all $j$ large enough. 
For notational simplicity, we omit the subscript $j$ in the computations below.

Define the barrier function $$\Phi(x_1,x')=\frac{1}{2}(\gamma+\Lambda^2d \eps)(x_1+\eps^2)^2-\frac{1}{2}\eps|x'|^2.$$
and notice that
$$\hat{\Phi}= \frac{1}{\eps}(\Phi-q)=\frac 12 d \, \Lambda^2 x_1^2 - \frac 12 |x'|^2 + O(\eps).$$
We compare $u$ and $\Phi$ on the boundary of the set $$U_r \cap \{x_1 \ge -\eps^2\}.$$ 
On $\{x_1=-\eps^2\}$ we have $ u \ge 0 \ge \Phi$. 
On the remaining part $\partial U_r \cap \{x_1 \ge -\eps^2\}$ we have $u \ge \Phi$ since
$\hat u > \hat \Phi$ for all small $\eps$.
In conclusion, $\Phi\le u$ along the boundary, and Proposition \ref{ComparisonPrinciple} gives  $u\ge\Phi$ in the interior of the domain. 

 In particular $u(0)\ge \Phi(0)>0$, contradicting $0\in\Sigma(u).$   Therefore we have  $a_{\pm}\le 0.$ 

Next we show that $a_\pm$ cannot be negative. Suppose that 
 $$a_+<0,$$ and in this case, we use a barrier to prove that $\{u=0\}$ contains a cone with positive opening and with vertex at $0$. With Proposition \ref{VanishingThickness}, this contradicts $0\in\Sigma(u).$

Since $a_- \le 0$, we can choose $r$ small such that
$$
\hat{u} < \frac{1}{2} a_+ x_1 - \frac 12 d \, \Lambda^2 x_1^2 + \frac 18 r ^2,
$$
near $\partial U_r$. We compare $u$ and $\Phi$ on the boundary of the set $U_r$ where
$$\Psi(x)= \frac{1}{2}\left(\gamma- \Lambda^2(d-1) \eps\right)\left(x_1+ A \eps \right)^2+\frac{1}{2}\eps|x'-\xi'|^2,$$
with $A:=a_+/(2 \gamma)$ and $|\xi'| \le r/2$. Since
$$\hat \Psi= \frac{1}{2} a_+ x_1 - \frac 12 (d-1) \, \Lambda^2 x_1^2 + \frac 12 |x'-\xi'|^2 + O(\eps),$$
we find that $\hat \Psi > \hat u$, hence $\Psi>u$ on $\partial U_r$ for all $\eps$ small.

From here Proposition \ref{ComparisonPrinciple} becomes applicable and gives $u\le\Psi$ in $U$. 

In particular this gives $$u(-A \eps,\xi')=0 \text{ for all $|\xi'|\le r/2.$}$$

Now note that with $u\in\SPE$, we have $D^2 u\ge -c_0\eps$ in $B_1$. Also since $e_1$ is the direction corresponding to the largest eigenvalue of $D^2p$, there is a cone of directions around $e_1$, say, $K\subset\Sph$ with a universal positive opening such that $\Dee p> c>0$ for all $e\in K$.  For small $\eps$ we can then apply Lemma \ref{Convexity} to get $$\Dee u\ge 0 \text{ in $B_{1/2}$}$$ for all $e\in K$. 

Together with $u(0)=0$ and $u(-A\eps, \xi')=0$ for all $|\xi'|\le r/2$, this implies that the coincidence set $\{u=0\}$ contains a cone of positive opening with vertex at $0$, contradicting Proposition \ref{VanishingThickness}.

This finishes the proof of \textit{Step 1}.

\

\textit{Step 2: Possibility (2) in Proposition \ref{ExpansionForThinObstacle} does not happen for $\hat u$.}

Suppose it happens, then for some $r>0$ and $\nu\in\Sph\cap\{x_1=0\}$,  we have $$D_\nu\hat{u}> 0 \text{ in $B_r\cap\{x_1\neq0\}$.}$$

Therefore, there is some $\sigma>0$ such that $$D_\nu\hat{u}\ge4\sigma \text{ in $B_{r}\cap\{|x_1|\ge\frac{1}{16\sqrt{\Lambda}}r\}$.}$$
Note that $|u-p_j|\le\eps_j$ in $B_1$ implies $\{u_j=0\}\subset\{|x_1|\le C\sqrt{\eps_j}\}$ for some universal $C$, we have $\{u_j>0\}\cap B_1\to B_1\backslash\{x_1=0\}$. 

Meanwhile, $\hat{u}_j\to\hat{u}$ locally uniformly in $B_1$ with $G_j(D^2\hat{u}_j)=0$ in $\{u_j>0\}$, where $$G_j(M)=\frac{1}{\eps_j}(F_j(\eps_j M+D^2q_j)-1).$$Consequently $\hat{u}_j\to\hat{u}$ in $C^{1,\alpha}_{loc}(B_1\backslash\{x_1=0\})$. Therefore, for large $j$, $$D_\nu\hat{u}_j\ge2\sigma \text{ in $B_{r}\cap\{|x_1|\ge\frac{1}{16\sqrt{\Lambda}}r\}$}.$$ That is, 
$$D_\nu u_j\ge2\sigma\eps_j  \text{ in $B_{r}\cap\{|x_1|\ge\frac{1}{16\sqrt{\Lambda}}r\}$}.$$
 
With $|u_j-q_j|\le C\eps_j$, in $\{u_j>\frac{1}{256}r^2\}$, we have 
\begin{equation*}
\frac{1}{2}\gamma x_1^2\ge \frac{1}{256}r^2-C\eps_j\ge \frac{1}{512}r^2
\end{equation*} for $j$ large. 
Consequently, $$\{u_j>\frac{1}{256}r^2\}\subset\{|x_1|\ge\frac{1}{16\sqrt{\Lambda}}r\}.$$ Thus, we have established 
$$D_\nu u_j\ge 2\sigma\eps_j  \text{ in $B_{r}\cap\{u_j>\frac{1}{256}r^2\}$}.$$

Also, by Lemma \ref{UniformLipschitz}, we have $D_\nu u_j\ge-L\eps_j$ in $B_r$.

Now take $\eta$ depending on $K=L$,  $r$ and $\sigma$ as in Lemma \ref{Monotonicity2}. 

Note the convergence of  $\hat{u}_j\to\hat{u}$ in $C^{1,\alpha}_{loc}(B_1\backslash\{x_1=0\})$  implies $$D_\nu u_j>0 \text{ in $B_r\cap\{|x_1|>\frac{1}{2}\eta\}$}$$ for large $j$. 

By the $C^{1,\alpha}$-regularity of $u_j$, there is a cone of directions $\hat{K}\subset\Sph$ around $\nu$ with positive opening such that for all $e\in\hat{K}$, we have $$D_e u_j\ge -K\eps_j \text{ in $B_r$,}$$$$D_e u_j\ge\sigma\eps_j \text{ in $B_r\cap\{u_j>\frac{1}{256}r^2\}$,}$$and $$D_e u_j>0 \text{ in $B_r\cap\{|x_1|>\eta\}.$}$$

Thus Lemma \ref{Monotonicity2} applies and gives $$D_eu_j\ge 0 \text{ in $B_{r/2}$ for all $e\in\hat{K}$}.$$

With $u_j(0)=0$, $u_j\ge 0$, this implies that $\{u_j=0\}$ contains a cone in $B_{r/4}$ in direction $-\hat{K}$, again contradicting Proposition \ref{VanishingThickness}. 

\

\textit{Step 3: Improved quadratic approximation, i.e. we show that
$$\mbox{$ \forall \delta>0$, \quad $ \exists r>0$, \, \,  $p_j' \in \mathcal Q$ such that $|u_j-p_j'| \le 3 \delta \eps_j \, r^2$ in $B_r$.} $$}
After the previous two steps, we have that the limiting profile $\hat{u}$ falls into  possibility (3) as in Proposition \ref{ExpansionForThinObstacle}. Consequently, for some $\delta>0$ to be chosen later, there is $r>0$ such that  $$|\hat{u}-\frac{1}{2}x\cdot Ax|<\delta r^2 \text{ in $B_r$,}$$
where $trace(A)=0$, and $e\cdot Ae\ge 0$ for all $e\in\Sph\cap\{x_1=0\}.$

Locally uniform convergence of $\hat{u}_j\to\hat{u}$ gives for large $j$ \begin{equation*}
|u_j-q_j-\eps_j\frac{1}{2}x\cdot Ax|<2\eps_j\delta r^2 \text{ in $B_r$.}
\end{equation*} 
Here  we omit the index $j$ for the sake of simplicity, then we have
 \begin{equation}\label{QuadApprox2}
|u-q-\eps\frac{1}{2}x\cdot Ax|<2\delta\eps r^2 \text{ in $B_r$.}
\end{equation}  

With Cauchy-Schwarz inequality and $e\cdot Ae\ge 0$ for all $e\in\Sph\cap\{x_1=0\},$ we see that there is a constant $C$, depending on $|A|$, such that 
 \begin{equation}\label{66}D^2q+A\eps+C\eps^2 \ge \frac{1}{4}\gamma e_1\otimes e_1\end{equation} for $\eps$ small.

Now note that we are assuming, after necessary scaling, that $F_{ij}(D^2q)=\delta_{ij}$, where $F_{ij}$ is the derivative of $F$ in the $(i,j)$-entry. Thus $trace(A)=0$ implies $$|F(D^2q+ A\eps+C\eps^2I)-1|\le C\eps^{1+\alpha_F}$$ by assumption \eqref{SecondAssumption}.

Consequently, there is $t\in [-C,C]$ such that the polynomial $$p'(x)=q+\frac{1}{2}\eps x\cdot Ax+\frac{1}{2}C\eps^2|x|^2+\frac{1}{2}t\eps^{1+\alpha_F}x_1^2$$ solves $$F(D^2p')=1.$$ Meanwhile, by \eqref{66} we have  
\begin{align*}D^2p'&=D^2q+A\eps+C\eps^2I+t\eps^{1+\alpha_F}e_1\otimes e_1\\&\ge\frac{1}{4}\gamma e_1\otimes e_1-C\eps^{1+\alpha_F}e_1\otimes e_1\\&\ge 0
\end{align*}for $\eps$ small. Thus $D^2p'\ge 0$ and $p'\in\Qua.$

Finally \eqref{QuadApprox2} implies that in $B_r$, 
\begin{align*}|u-p'|&\le |u-q-\eps\frac{1}{2}x\cdot Ax|+ C\eps^2 r^2+\frac{1}{2}|t|\eps^{1+\alpha_F}x_1^2\\&\le 2\delta\eps r^2+C\eps^{1+\alpha_F}r^2
\\ &\le 3\delta\eps r^2,
\end{align*}for all $\eps$ small.

\

\textit{Step 4: Improved convexity, i.e. we show that 
$$ u \in \mathcal{S}(p',(1-\beta)\eps,r),$$
with $\beta>0$ depending on $\kappa$ and universal constants.}

It remains to show that (see \eqref{convexity}) $$D^2u\ge-c_0(1-\beta)\eps \, I \text{ in $B_r$,}$$ where $c_0$ is the constant as in Definition \ref{ApproxClass}.

For fixed $e\in\Sph$, define $$w=\Dee u+c_0\eps.$$ Then $u\in\SPE$ implies $w\ge 0$ in $B_1$.

Now for small $\eps$, $u>0$ in $B_{\frac{1}{4}r}(\frac{1}{2}re_1)$. Thus Proposition \ref{EstimateForDifference} implies $$|\Dee u-\Dee p'|\le C\delta\eps r^2 \quad \text{in $B_{\frac{1}{8}r}(\frac{1}{2}re_1)$}$$ for a universal $C$. 

Now fix $\delta$, depending on universal constants and $\kappa$, such that the right-hand side is less than $\frac{1}{2}c_0\eps$. Then 
$$|u-p'|<\frac{1}{2}\eps r^2 \text{ in $B_r$}$$ and $$\Dee u\ge-\frac{1}{2}c_0\eps \text{ in $B_{\frac{1}{8}r}(\frac{1}{2}re_1)$}.$$

In particular \begin{equation}\label{2}w\ge\frac{1}{2}c_0\eps \text{ in $B_{\frac{1}{8}r}(\frac{1}{2}re_1)$}.\end{equation}

If we solve $$\begin{cases}\mathcal{M}^+_\Lambda(\Phi)=0 &\text{ in $B_1\backslash B_{\frac{1}{8}r}(\frac{1}{2}re_1),$}\\ \Phi=0 &\text{ along $\partial B_1$,}\\ \Phi=\frac{1}{2}c_0 &\text{ along $ B_{\frac{1}{8}r}(\frac{1}{2}re_1)$,}\end{cases}$$where $\mathcal{M}^+_\Lambda$ is the maximal Pucci operator \cite{CC}, then $$w\ge\eps\Phi \text{ along $\partial B_1\cup\partial B_{\frac{1}{8}r}(\frac{1}{2}re_)$}.$$

Meanwhile, along $\partial\PosS$, \begin{equation}\label{3}w\ge c_0\eps\end{equation} by \eqref{Contact}. Thus $$w\ge\eps\Phi \text{ along $\partial\PosS.$}$$

In conclusion, $$w\ge\eps\Phi \text{ along  $\partial (\PosS\cap B_1\backslash B_{ \frac{1}{8}r}(\frac{1}{2}re_1))$}.$$
With $L_u(w)\le 0$ in $\PosS$, comparison principle gives  $$w\ge\eps\Phi \quad \text{ in $\PosS\cap B_1\backslash B_{ \frac{1}{8}r}(\frac{1}{2}re_1))$.}$$ Together with \eqref{2} and \eqref{3}, this implies   $$w\ge\eps\Phi \text{ in $B_1$.}$$

Meanwhile, there is a constant $\beta'\in(0,1)$, depending on universal constants and $\kappa$,  such that $$\Phi\ge\frac{1}{2}\beta' c_0 \quad \text{in $B_r$}.$$Thus $$D_{ee}u= w-c_0\eps\ge c_0\eps(-1+\frac{1}{2}\beta') \text{ in $B_r$.}$$ Define $\beta=\frac{1}{2}\beta'$, then in $B_{r}$, we have $$D^2u\ge -(1-\beta)c_0\eps \, I,$$ and $$|u-p'|\le(1-\beta)\eps r^2.$$ That is, $$u\in\mathcal{S}(p,(1-\beta)\eps, r)$$ for $p'\in\Qua$. 

This contradicts our construction of $u$ at the beginning of this proof. 
\end{proof} 

%%%%%%%%%%%%%%%%%%%%%%%%%%%%%%%%%%%%%%%%%%%%%%%%%%%%%%%%%%%%%%%%%%%%%%%%%%%%%%%%%%%%%%%%%%%%%%%%%%%%%%%%%%%%%%%%%%%%%%%%%%%%%%%%%%%%%%%%%%%%%%%%%%%%%%%%%%%%
\section{Quadratic approximation of solution: Case 2}
In this section, we prove a version of Lemma \ref{QuadraticApproxTopStratum} but for $u\in\SPE$ where $\lambda_2(D^2p)\gg\eps.$ 
Here the situation is different since the zero set $\{u=0\}$ concentrates around subspaces of codimension at least 2, say
\begin{equation}\label{n-k}
\{x'=0 \in \R^k\},  \quad k \ge 2, \quad \mbox{where} \quad x':=(x_1,..,x_k), \quad x'':=(x_{k+1},...,x_d).
\end{equation} This brings technical challenges as the normalized solution $\hat u=\frac 1 \eps (u-p)$ now solves an obstacle problem with an obstacle $\hat O=-\frac 1 \eps p$ whose capacity converges to $0$ as $\eps \to 0$. 

We define $h$ to be the solution to the unconstrained problem
\begin{equation}\label{hdef}
\begin{cases}
F(D^2h)=1 &\text{ in $B_1$,}\\ h=u &\text{ on $\partial B_1$.}
\end{cases}
\end{equation}
We will show that $\hat u$ is well approximated in $L^\infty$ by the corresponding function $\hat h$,
$$\hat h := \frac 1 \eps (h-p),$$
but only away from a tubular neighborhood around the $(d-k)$-dimensional subspace above (see Lemma \ref{OutsideCylinder}). Inside this neighborhood, the difference between $\hat h$ and $\hat u$ could be of order 1, and $\hat u$ has no longer a uniform modulus of continuity (as $\eps \to 0$) in $B_{1/2}$ as in the codimension 1 case. 

Heuristically, as $\eps \to 0$, we end up with limiting functions $\bar u$, $\bar O$ and $\bar h$, that satisfy that $|\bar h|$, $|\bar u|$ and $\max \bar O$ are all bounded by 1 in $B_1$, and

\noindent 1) $\bar h$ is a solution to a constant coefficient elliptic equation, 

\noindent 2) the obstacle $\bar O$ is a concave quadratic polynomial supported on the $x''$-subspace, extended to $-\infty$ outside its support, 

\noindent 3) $\bar u=\max\{\bar h, \bar O \}$, which can be discontinuous.

The improved quadratic error for $\hat u$ cannot be deduced right away from the $C^{2,\alpha}$ estimate of $\bar h$ at the origin. This will follow after we show that $0 \in \Sigma(u)$ essentially implies that $\bar h$ and $\bar O$ are tangent of order 1 at the origin in the $x''$ direction and $\bar O$ can only separate on top of $\bar h$ in this direction by a small quadratic amount. 

It turns out that the improvement in convexity and approximation is much slower. Instead of $\eps\to(1-\beta)\eps$ as in Lemma \ref{QuadraticApproxTopStratum}, we only have an improvement of the form $\eps\to(\eps-\eps^\mu)$, where $\mu>1,$ universal. This is consistent with $C^{1,\log^{\eps}}$-regularity of covering for lower strata in the classical obstacle problem \cite{CSV}. 

This slow rate of improvement could a priori break the convergence of the polynomials $p_k$ and the uniqueness of the blow-up profile, as well as the iteration scheme. Suppose $p_k$ is the approximating quadratic polynomial in the $k$th iteration. Then a rate of $\eps\to(1-\beta)\eps$ implies $$|D^2p_{k+1}-D^2p_{k}|\le C(1-\beta)^k\eps_0.$$ The summability of this sequence implies the convergence of $D^2p_k.$ When the rate is $\eps\to(\eps-\eps^\mu)$, this is not true anymore. 

In the next section, we establish the convergence of  $D^2p_k$  by working instead with the corresponding approximations $D^2h_k(0)$. These are not necessarily positive definite, but still approximate $u$ quadratically with error proportional to $\eps_k$. The main point is that the series
\begin{equation}\label{sumh}
\sum |D^2 h_k(0)-D^2 h_{k+1}(0)|
\end{equation}
is convergent, which is a consequence of the main result of this section, Lemma \ref{QuadraticApproxLowerStratum} below. This lemma provides a dichotomy concerning the rate of the quadratic improvement between two consecutive balls.
Essentially it says that either we have a fast improvement as in Lemma \ref{QuadraticApproxTopStratum}, or the difference between consecutive errors $\eps_k$ is bounded below by the difference between $u$ and $h$ at some point away from the $x''$ subspace, which could be as small as $\eps^\mu$. 

We recall that by Definition \ref{ApproxClass}, $u\in\SPE$ means that $u$ solves \eqref{OP}, and
$$|u-p|\le\eps r^2 \quad \mbox{and} \quad D^2u\ge-c_0\eps I \quad \mbox{in $B_r$.}$$ 
Here $c_0=1/(16 \Lambda^2)$ and $p \in \UQua$ (see Definition \ref{QuadraticSolution}), which means that $p$ is a convex quadratic polynomial that satisfies $p(0)=0$, $F(D^2p)=1$.

\begin{lem}[Quadratic approximation of solution: Case 2]\label{QuadraticApproxLowerStratum}
Suppose $u\in\SPE$ with $0\in\Singu$ and $p\in\UQua.$ 
There are universal constants $\kappa_0$ large, $\bar{\eps}$ small, and $\rho\in(0,1/2)$ such that if $\eps<\bar{\eps}$  and  $$\lambda_2(D^2p)\ge\kappa_0\eps,$$ then $$u\in \mathcal{S}(p',\eps',\rho)$$ for some $p'\in\UQua,$ and one of two alternatives happens for $\eps'$:
\begin{enumerate}
\item{$$\eps'\le(1-\beta)\eps \quad \mbox{ for a universal $\beta\in(0,1)$; or}$$}
\item{$$\eps'\le\eps-\eps^\mu \quad \mbox{ and} \quad  (u-h)(\frac{1}{2}\rho e_1)\le C(\eps-\eps'),$$ 
for some universal constants $\mu, C>1,$ where $h$ is the solution  to \eqref{hdef}.}
\end{enumerate}
\end{lem}

The dichotomy is dictated by the behavior of the matrix $D^2 h(0)$ along the $x''$ subspace. If $D^2_{x''}h(0) \ge -\frac{c_0}{8} \eps \, I,$ then we end up in alternative (1), otherwise we end up in (2).

Suppose that we are in the slow improvement situation (2). Let $h'$ denote the solution to \eqref{hdef} in the ball $B_\rho$. By maximum principle, we have $u \ge h' \ge h$ in the common domain. The Harnack inequality for the difference $h'-h$ and Proposition \ref{EstimateForDifference} imply
\begin{align*}
|D^2h'(0)-D^2 h(0)| & \le C_\rho \|h'-h\|_{L^\infty} \\
& \le C (h'-h)(\frac{1}{2}\rho e_1) \le C (u-h)(\frac{1}{2}\rho e_1)\le C_1(\eps-\eps') ,
\end{align*}
for some universal $C_1$.  Iteratively,  the series in \eqref{sumh} is  bounded from above by a telescoping sum. Thus its convergence is justified.

Recall that $\lambda_2(M)$ denotes the second largest eigenvalue of a matrix $M$. A technical point is that we are working in the class of quadratic polynomials $p \in \UQua$ defined in Definition \ref{QuadraticSolution} which have a linear part as well.

Up to a rotation, $p$ takes the form $$p(x)=\frac{1}{2}\sum a_jx_j^2+\sum b_jx_j$$ with $a_1\ge a_2\ge\dots\ge a_d\ge 0$, $a_2\ge\kappa_0\eps$ and $F(\sum a_j e_j\otimes e_j)=1.$ Throughout this section we assume that $p$ is of this form. 

Then by $p\ge u-\eps\ge-\eps$ in $B_1$ we have $$\frac{1}{2}a_jx_j^2+b_jx_j\ge -\eps$$ for $x_j\in[-1,1]$ and $1\le j\le d.$  If $a_j\ge 2\eps$, we have \begin{equation}\label{bAnda}|b_j|\le \sqrt{2a_j\eps}.\end{equation}

For a positive constant $\eta$, we define the following cylinder $$\mathcal{C}_\eta=\{|(x_1,x_2)|\le\eta\}.$$ We first show that $u$ is well approximated by $h$ outside this cylinder. 

\begin{lem}\label{OutsideCylinder}
Let $u, p, h$ be as in  Lemma \ref{QuadraticApproxLowerStratum}.

Given $\eta$ small, there is $\kappa_\eta$, depending on universal constants and $\eta$, such that if $a_2\ge\kappa_\eta\eps$, then $$\|u-h\|_{C^2(B_{1/2}\backslash\mathcal{C}_{\eta})}\le\eta\eps$$ for all $\eps$ small, depending on $\eta.$ 
\end{lem} 

\begin{proof}
Let $0<\eta'\ll\eta$ to be chosen, depending on $\eta$. There is $\kappa_{\eta'}$ large such that $|u-p|<\eps$ and $a_2\ge\kappa_{\eta'}\eps$ imply $$u>0 \text{ outside $\mathcal{C}_{\eta'}$.}$$
Consequently, Proposition \ref{EstimateForDifference} gives \begin{equation}\label{HessianClose}D^2u\to D^2p \text{ locally uniformly  in $B_{1}\backslash\mathcal{C}_{\eta'}$}\end{equation} as $\eps\to0.$

Now let $\varphi\ge 0$ be a smooth function such that $\varphi=2$ in $\mathcal{C}_{2\eta'}$, $\varphi=0$ outside $\mathcal{C}_{3\eta'}$ and $|\nabla\varphi|\le 2/\eta'.$ 

We solve the following equations:
$$\begin{cases}
L_p(v)=0 &\text{ in $B_1\backslash\mathcal{C}_{2\eta'}$,}\\ v=\varphi &\text{ along $\partial(B_1\backslash\mathcal{C}_{2\eta'})$;}
\end{cases}$$ and 
$$\begin{cases}
L_u(\tilde{v})=0 &\text{ in $B_1\backslash\mathcal{C}_{2\eta'}$,}\\ \tilde{v}=\varphi &\text{ along $\partial(B_1\backslash\mathcal{C}_{2\eta'})$.}
\end{cases}$$

Since $v$ solves a constant coefficient equation and vanishes on $\partial B_1$ outside the thin cylinder $C_{3\eta'}$ around the codimension two set $\{x_1=x_2=0\}$, there is a modulus of continuity $\omega(\cdot)$, depending only on $\eta$, $\Lambda$, $d$, such that 
$$0\le v\le\omega(\eta') \text{ in $B_{3/4}\backslash\mathcal{C}_{\frac{1}{2}\eta}$.}$$
We use \eqref{HessianClose} together with Proposition \ref{v-w}, and estimate
$$|v-\tilde{v}|\le\tilde{\omega}(\eps) \text{ in $B_{3/4}\backslash\mathcal{C}_{\frac{1}{2}\eta}$,}$$
with $\tilde{\omega}(\eps)$ a modulus of continuity which depends also on $\eta'$. 

Note that $u>0$ outside $\mathcal{C}_{\eta'}$, and we have $L_u(u-h)\ge 0$ in $B_1\backslash\mathcal{C}_{2\eta'}$. Consequently, comparison principle gives $$u-h\le\eps\tilde{v}\le\eps(\omega(\eta')+\tilde{\omega}(\eps)) \text{ in $B_{3/4}\backslash\mathcal{C}_{\frac{1}{2}\eta}.$}$$ On the other hand, we always have $u-h\ge 0$. Therefore, Proposition \ref{EstimateForDifference} gives $$\|u-h\|_{C^2(B_{1/2}\backslash\mathcal{C}_{\eta})}\le C(\eta)\eps(\omega(\eta')+\tilde{\omega}(\eps)).$$ From here, we first choose $\eta'$ such that $C(\eta)\omega(\eta')<\frac{1}{2}\eta$, and then choose $\eps$ such that $C(\eta)\tilde{\omega}(\eps)<\frac{1}{2}\eta$. This gives the desired estimate.
 \end{proof} 

We now give the proof of the main result in this section:
\begin{proof}[Proof of Lemma \ref{QuadraticApproxLowerStratum}]

As discussed above, we define the normalizations $$\hat{u}=\frac{1}{\eps}(u-p),  \text{ }\hat{h}=\frac{1}{\eps}(h-p)  \text{ and }\hat{O}=\frac{1}{\eps}(0-p).$$
Then in $B_1$, we have $$-1\le\hat{h}\le\hat{u}\le 1, \quad \quad \hat u(0)=\hat O(0)=0,$$
and Proposition \ref{EstimateForDifference} implies \begin{equation}\label{hC2}\|\hat{h}\|_{C^{2,\alpha}(B_{3/4})}\le C\end{equation} for some universal constant $C$.

We divide the technical proof into 6 Steps. Here we give an outline first. 

We decompose the space $x=(x',x'')$ according to  the curvatures of the obstacle $\hat O$. The curvatures are very negative along the directions in the $x'$-subspace, and are uniformly bounded in the $x''$-subspace. In Steps 1-2 we show that $\hat h$ and $\hat O$ are ``essentially tangent" in the $x''$ direction at the origin, and deduce that $\hat O$ can only slightly separate quadratically on top of $\hat h$ near the origin. In Step 3, we show that the same is true for $\hat u$. In Step 4, we use the $C^{2,\alpha}$ estimate for $\hat h$ to approximate $u$ quadratically in $B_\rho$ by a polynomial $p' \in \UQua$ with an improved error $\frac \eps 2$. The convexity estimate for $D^2 u$ in $B_\rho$ (see \eqref{convexity} in Definition \ref{ApproxClass}) is given in Steps 5 and 6, according to whether or not the obstacle $\hat O$ separates quadratically on top of $\hat h$ along some direction in the $x''$ subspace. This leads to our dichotomy.

Throughout this proof, there are several parameters to be fixed in the end. 

The radius $\rho\in(0,1/2)$ depends only on universal constants. The parameter $\delta>0$ can be made arbitrarily small, and will be chosen to be universal. The parameter $\eta$ from Lemma \ref{OutsideCylinder}, which depends on $\delta$, allows us to make $\hat{u}$ and $\hat{h}$ very close to each other. This $\eta$ imposes the choice of $\kappa_0=\kappa_\eta$ as in Lemma \ref{OutsideCylinder}. The parameter $\bar{\eps}$ is chosen after all these. 

We introduce some notations. For $\delta$ small to be chosen,  let $k\in\{1,2,\dots, d\}$ be such that 
\begin{equation}\label{a_k}
a_k\ge 2\delta^{-4}\eps>a_{k+1}.
\end{equation} 
Then we decompose the entire space $\R^d$ as $x=(x',x'')$, where $x'=(x_1,x_2,\dots,x_{k})$ and $x''=(x_{k+1},x_{k+2},\dots, x_d).$

The obstacle $\hat O$ is changing rapidly in the $x'$ direction, and we denote by $\minx$ the point in this direction where its maximum is achieved, which is the same as the minimum point for $p$ in the $x'$ direction. 

Precisely, let $\underline{x}'$ be the minimum point of $x'\mapsto p(x',0)$. Then by \eqref{bAnda} and \eqref{a_k}, we have that $\underline x'$ is close to the origin
 $$|\underline{x}'|\le\delta^2, \quad \quad \mbox{and} \quad -\eps\le p(\underline{x}',0)\le0.$$
 
 We write $p$ as the sum of two quadratic polynomials in the $x'$ and $x''$ variables   
 $$p(x',x'')=p_1(x'-\minx) - p_1(\minx)+p(0,x''),$$ where $p_1 \ge 0 $ is the homogenous of degree 2 polynomial
 $$p_1(x')=\frac{1}{2}\sum_{j\le k}a_jx_j^2,$$ 
 %and $$p_2(x'')=\frac{1}{2}\sum_{j> k}a_jx_j^2+\sum_{j> k}b_jx_j.$$
The obstacle $\hat O$ satisfies
$$ |\nabla_{x''}\hat O|, |D_{x''}^2 \hat O| \, \le C_\delta, \quad \hat O ((\minx,0)) \ge 0.$$

 \
 
\textit{Step 1: If $\eta$ and $\eps$ are small depending on $\delta$, then }
\begin{equation}\label{step1}
|\nabla_{x''}(\hat{h}-\hat{O})(0)|<\delta.
\end{equation}

The idea is to show that otherwise $u$ is monotone in a cone of directions near the $x''$ subspace, and we contradict $0 \in \Sigma(u)$.

Suppose  there is $i>k$ such that $D_i(\hh-\hat{O})(0)>\delta.$ 

By $|D_{ii}\hat{O}|\le C_\delta$ and the universal estimate \eqref{hC2},  we have $$D_i(\hh-\hat{O})>\frac{1}{2}\delta \text{ in $B_{r}(0)$}$$ for some $r>0$ depending only on $\delta.$

Meanwhile, since $D^2u\ge-c_0\eps$ and $|D^2_{x''}p|\le C_\delta\eps$ in $B_1$, we have $$D_{ii}\hat{u}\ge-2C_\delta \text{ in $B_1.$}$$ Together with $|\hat{u}|\le 1$, this implies $$|D_i\hu| \le \frac 12 C_\delta \quad \Longrightarrow \quad D_i(\hu-\hat{O}) \ge - \frac 32 C_\delta \quad \text{ in $B_{1/2}$.}$$

By continuity, there exists of a cone of directions, $\hat{K}\subset\Sph$, with positive opening around $e_i$, such that for all $e\in\hat{K}$, $$D_e(\hu -\hat{O})\ge -2C_\delta \text{ in $B_{1/2}$,}$$
and $$D_e(\hh-\hat{O})>\frac{1}{4}\delta \text{ in $B_{r}(0)$}.$$

Define the constant $\bar{\eta}$ as in Lemma \ref{Monotonicity1} depending on $r$, $K=-2C_\delta$, and $\sigma=\frac{1}{8}\delta$.  If we choose $\eta<\min\{\bar{\eta},\frac{1}{8}\delta\},$ then Lemma \ref{OutsideCylinder} gives $$D_e u\ge \eps(D_e(\hh-\hat{O})-\eta)\ge \frac{1}{8}\delta\eps \text{ in $B_r\backslash\mathcal{C}_\eta$,}$$ and 
$D_e u\ge-2C_\delta\eps \text{ in $B_{r}$.}$  Lemma \ref{Monotonicity1} gives $$D_eu\ge 0 \text{ in $B_{r/2}$}$$  for all $e\in\hat{K}.$

This implies that $\{u=0\}$ contains a cone of positive opening with vertex at $0$, contradicting Proposition \ref{VanishingThickness}.

\

\textit{Step 2: If $\eta$ and $\eps$ are small depending on $\delta$, then} 
\begin{equation}\label{step2}
|\hat{h}(\minx,0)-\hat{O}(\minx,0)|<\delta.
\end{equation}

If $\hat O$ is a bit larger than $\hat h$ at $(\minx,0)$ we show that $\hat u$ coincides with the obstacle $\hat O$ in a small neighborhood of $(\minx,0)$ hence $u=0$ in this neighborhood. On the other hand, by Lemma \ref{Convexity}, $u$ is convex in the directions close to the $x'$ subspace, and we obtain that $\{u=0\}$ contains a cone with vertex at the origin and reach a contradiction. Next we provide the details.

Note that $\hh(0)\le\hu(0)=0$ and $\hat{O}(\minx,0)\ge 0,$ and then the upper bound for $\hat h - \hat O$ at $(x',0)$ follows from $|\minx|\le\delta^2$ and \eqref{hC2}.

Suppose $\hat{h}(\minx,0)-\hat{O}(\minx,0)<-\delta$. 

We show that $u=0$ in a neighborhood of $(\minx,0)$ by using barriers.

Since $\hat{h}$ has universal Lipschitz norm, and $|D_{x''}\hat{O}_2| \le C_\delta$, there is $r>0$ depending on $\delta$, such that $$\hat{h}(x)<\hat{O}(\minx,x'')-\frac{1}{2}\delta \text{ in $B_r(\minx,0)$.}$$
That is, $$h(x)<p_1(x'-\minx)-\frac{1}{2}\delta\eps \text{ in $B_r(\minx,0)$.}$$

Consequently for $\eta$ small, Lemma \ref{OutsideCylinder} implies $$u(x)<p_1(x'-\minx)-\frac{1}{4}\delta\eps \text{ in $B_r(\minx,0)\backslash\mathcal{C}_\eta$}.$$ 
%In particular,  with $p_1(0)=0$ and $u\ge 0$, this implies that there is a neighborhood of $(\minx,0)$ entirely contained in $\mathcal{C}_\eta.$

Let $\Omega:=\{|x_1|<\eta\}\cap B_r(\minx,0)$. We define the barrier  
$$\Psi(x',x''):=p_1(x'-\minx)+\frac{1}{2}B\eps(|x-(\minx,0)|^2-2\Lambda^2(x_1-\underline{x}'_1)^2),$$
 for some $B$ depending on $\delta$ and $r$.
Note that $\Psi(\minx,0)=0$, and $$\Psi\ge p_1(x'-\minx)-\Lambda^2B\eps(x_1-\underline{x}'_1)^2\ge \frac{1}{2}(a_1-2\Lambda^2B\eps)(x_1-\underline{x}'_1)^2,$$ thus $\Psi\ge 0$ if $\eps$ is small. We choose $B$ large such that by ellipticity \eqref{Ellipticity}
$$F(D^2\Psi) \le F(D^2 p_1) - B \eps \Lambda \le F(D^2 p)=1, $$
and on $\partial \Omega \cap \partial B_r(\minx,0)$, for $\eta$ sufficiently small, 
\begin{align*}
\Psi&\ge p_1(x'-\minx)+\frac{1}{2}B\eps(r^2-2\Lambda^2\eta^2)\\&\ge p_1(x'-\minx)+\frac{1}{4}B\eps r^2\\&\ge p_1(x'-\minx)+p(\minx,0)+p(0,x'')+2\eps=p(x)+2\eps.
\end{align*}
Consequently, $$\Psi\ge u+\eps \text{ on $\partial \Omega \cap \partial B_r(\minx,0)$.}$$

Meanwhile, on $\{|x_1|=\eta\} \cap B_r(\minx,0)$,
\begin{align*}u&<p_1(x'-\minx)-\frac{1}{4}\delta\eps\\&=\Psi-\frac{1}{4}\delta\eps-\frac{1}{2}B\eps(|x-(\minx,0)|^2-2\Lambda^2(x_1-\underline{x}'_1)^2)\\&\le \Psi-\frac{1}{4}\delta\eps+B\eps\Lambda^2\eta^2.
\end{align*} Thus if $\eta$ is small, then $$u<\Psi-\frac{1}{8}\delta\eps \text{ on $\partial \Omega$.}$$

Consequently, we can apply Proposition \ref{ComparisonPrinciple} to $\Psi$ in $\Omega$ to get $u\le \Psi$ in $\Omega.$

In particular, we have $u(\minx,0)=0.$

Now note that along $\partial \Omega$, we have $u<\Psi-\frac{1}{8}\delta\eps$. Thus we can translate $\Psi$ a small amount and still preserve the comparison $u\le\Psi$ along $\partial\Omega$. This gives $$\{u=0\}\supset B_{r'}(\minx,0)$$ for a small $r'>0.$

With $a_j>2\delta^{-4}\eps$ for $j\le k$, we can apply Lemma \ref{Convexity} to get $$\Dee u\ge 0$$ for all directions in a cone around the subspace $\{(x',x'')|x''=0\}.$ With $u(0)=0$ and $\{u=0\}\supset B_{r'}(\minx,0)$, this generates a cone with positive opening and vertex at $0$ in $\{u=0\}$, contradicting Proposition \ref{VanishingThickness}.

\

\textit{Step 3:  For $\delta$ universally small, and $\eps,\eta$ small depending on $\delta$,  we have 
\begin{equation}\label{step3}\hu\le\hh+c_0|x''|^2+4\delta \text{ in $B_{1/4}$}.
\end{equation}}

The inequality holds outside $\mathcal C_\eta$ by Lemma \ref{OutsideCylinder}. It remains to establish it in $\mathcal C_\eta$.

First we use Steps 1 and 2 to show that a similar inequality holds for $\hat O$:
\begin{equation}\label{2041}
\hat O \le \hat h + c_0 |x''|^2 + 3 \delta =: g \quad \quad \mbox{in} \quad B_{1/2}.
\end{equation}
Then we use barriers to extended the inequality from $\hat O$ to $\hat u$.

Note that in $B_{1/2}\backslash\mathcal{C}_\eta$, Lemma \ref{OutsideCylinder} gives
$$D^2\hh-D^2\hat{O}\ge D^2\hat{u}-D^2\hat{O}-\eta \, I=\frac{1}{\eps}D^2u-\eta \, I \ge-(c_0+\eta) \, I.$$ By choosing $\eta$ small, and using that $D^2 \hat O$ is constant together with the H\"older continuity of $D^2 \hat h$ (see \eqref{hC2}), we extend the estimate to the full ball 
\begin{equation}\label{51}D^2(\hh-\hat{O})\ge-2c_0 \, I \text{ in $B_{1/2}$.}\end{equation}
Moreover, $D^2_{x'} \hat O \le - 2\delta^{-4} I$, $D^2_{x'x''}\hat O=0$ and $|D^2 \hat h| \le C$ universal imply 
\begin{equation}\label{52}D^2_{x'}(\hh-\hat{O})\ge\delta^{-4} \, I, \quad \quad |D^2_{x'x''}(\hat{h}-\hat O)|\le C \quad \mbox{in $B_{1/2}$.}\end{equation}  

Since $\minx$ is an extremal point of $x'\mapsto \hat{O}(x',0),$ $\nabla_{x'}\hat{O}(\minx,0)=0.$ Together with \eqref{hC2}, we have 
\begin{equation}\label{514}
|\nabla_{x'}(\hh-\hat{O})(\minx,0)|\le C.
\end{equation}
Also, the conclusion \eqref{step1} of Step 1 together with the second estimate in \eqref{52} and $|\minx| \le \delta^2$ give
\begin{equation}\label{515}
|\nabla_{x''}(\hh-\hO)(\minx,0)|\le 2\delta.
\end{equation}

Now it is easy to check that the estimates \eqref{51}-\eqref{515} together with the conclusion \eqref{step2} of \textit{Step 2} imply the claim \eqref{2041}.

Next we show $\hu\le g+\delta$ in $B_{1/4}\cap\mathcal{C}_\eta$ with $g$ defined in \eqref{2041}.

To this end, pick a point $x^*\in B_{1/4}\cap\mathcal{C}_\eta$ and $r>0$ depending on $\delta$ such that $$g(x^*)+\frac{1}{2}\delta>\max_{B_r(x^*)}\{g,\hh+\eta\}.$$

Define  $\Omega=B_r(x^*)\cap\{|x_1|<\eta\},$ and  $$v=g(x^*)+\delta+B(|x-x^*|^2- 2\Lambda^2|x_1-x_1^*|^2),$$ where $B$ is a large constant such that $Br^2>2.$ 

We compare $u$ and the barrier function $$\Psi(x',x''):=p(x',x'')+\eps v,$$
in the set $\Omega$. We have $v\ge g+\frac{1}{2}\delta-2\Lambda^2B\eta^2\ge g$ if $\eta$ is small, thus $$\Psi\ge p+\eps g\ge p+\eps\hO=0.$$

Along $\partial B_r(x^*)\cap\partial \Omega$, 
\begin{align*}
v&\ge g(x^*)+\delta+B(r^2-2\Lambda^2\eta^2)\\&\ge g(x^*)+\delta+2-2B\Lambda^2\eta^2\ge 1
\end{align*}if $\eta$ is small. Thus on $\partial B_r(x^*)\cap\partial \Omega$, $$\Psi\ge p+\eps\ge u.$$
Along $ B_r(x^*)\cap\{|x_1=\eta\}$, $v\ge \hh+\frac{1}{2}\delta+\eta\ge\hu$ by Lemma \ref{OutsideCylinder}. Again $\Psi\ge p+\eps\hu=u$.  

Since $F(D^2 \Psi) \le F(D^2p)=1$, we can apply Lemma \ref{ComparisonPrinciple} to $\Psi$ and $\Omega$ to get $\Psi\ge u \text{ in $\Omega$.}$ In particular, $u(x^*)\le v(x^*)=g(x^*)+\delta$, which is the desired estimate.

\

\textit{Step 4:
\begin{equation}\label{step4}
\exists \, \, p'\in\UQua \quad \mbox{such that} \quad |u-p'| \le \frac \eps 2 \, \rho^2 \quad \mbox{in} \quad B_\rho,
\end{equation}
with $\rho$ universal, provided that $\delta$ is chosen sufficiently small, depending on universal constants.}

Define $$q(x)=\frac{1}{2}x\cdot D^2 \hat h(0)x+\nabla \hat h(0)\cdot x.$$
Then \eqref{hC2} implies $$|\hh-\hh(0)-q|\le C\rho^{2+\alpha}\le c_0\rho^2 \text{ in $B_{\rho}$,}$$ if we choose $\rho$ universally small.  Here $c_0$ is the constant in Definition \ref{ApproxClass}.

With $\hh\le\hu\le\hh+c_0|x''|^2+4\delta$ in $B_{1/4}$ from \textit{Step 3} and $\hu(0)=0$, this implies 
$$|\hu-q|\le 2c_0\rho^2+8\delta \text{ in $B_\rho$.}$$ Fixing $\delta$ universally small such that $8\delta<c_0\rho^2$, then we have $$|u-p-q\eps|\le 3c_0\eps\rho^2 \text{ in $B_\rho$.}$$

Define $\tilde{p}=p+q\eps$, then $D^2\tilde{p}=D^2p+\eps D^2q=D^2h(0).$ Thus $F(D^2\tilde{p})=1.$

Next we perturb slightly $\tilde p$ into a convex polynomial $p' \in \UQua$. 

From \eqref{51} we know $D^2\tilde{p}\ge -2c_0\eps \, $. We denote by $M:=(D^2 \tilde p)^+$ the positive part of $D^2 \tilde p$, hence \eqref{Ellipticity} gives $$1\le F(M)\le 1+ 2 \Lambda c_0\eps, \quad \mbox{and} \quad \frac 1 \Lambda \le \|M\| \le 2 \Lambda.$$ 

Consequently, we can pick $t\in [0, 2c_0\Lambda^2 \|M\|^{-1} ]$ such that $$F((1-t \eps)M)=1.$$
Denote the new quadratic polynomial $$p'(x):=(1-t \eps)\, \frac{1}{2}x\cdot Mx+\nabla h(0)\cdot x.$$ 
Then clearly $p'\in\UQua,$ and 
$$|p'-\tilde{p}|\le ( \Lambda c_0 + \Lambda^2 c_0 + C \eps)\eps\rho^2 \quad \mbox{in} \quad B_\rho.$$ Thus, by recalling the definition of $c_0$ in Definition \ref{ApproxClass}, we have $$|u-p'|\le 3c_0\eps\rho^2+ 3 \Lambda^2 c_0\eps\rho^2\le \frac{1}{2}\eps\rho^2 \quad \mbox{in} \quad B_\rho.$$ 

Next we improve the convexity of $u$. 
There are two cases to consider, corresponding to the two alternatives as in  Lemma \ref{QuadraticApproxLowerStratum}.

\

\textit{Step 5: If $D^2_{x''}(\hh-\hO)(0)\ge-\frac{1}{8}c_0 \, I,$ then $D^2u\ge-(1-\beta)c_0\eps \, I $ in $B_\rho$, for some $\beta$ universal.}

The inequality at the origin can be extended to a fixed neighborhood by continuity and then, by Lemma \ref{OutsideCylinder} transferred to $D^2u$ away from the cylinder $\mathcal C_\eta$. This can be further extended to the whole domain by using that pure second derivatives of $u$ are global supersolutions.

More precisely, our hypothesis together with \eqref{hC2} and the fact that $D^2 \hat O$ is constant imply that the inequality holds in a small ball $B_c$, $c$ universal, with $-\frac{1}{4}c_0 \, I$ as the right hand side. Using \eqref{52}, we can extend the inequality to the full Hessian, $$D^2(\hh-\hO)\ge-\frac{3}{8}c_0 \, I \quad \mbox{in} \quad B_c.$$ 
By choosing $\eta$ small, Lemma \ref{OutsideCylinder} gives $$D^2u\ge-\frac{1}{2}c_0\eps \, I \text{ in $B_{c/4}(\frac{1}{2}ce_1)$}.$$ From here we can apply the same argument as in \textit{Step 4} of the proof for Lemma \ref{QuadraticApproxTopStratum} to get $$D^2u\ge-(1-\beta)c_0\eps \, I\text{ in $B_\rho.$}$$ This corresponds to the first alternative as in Lemma \ref{QuadraticApproxLowerStratum}.

\

\textit{Step 6: If $D_{\xi \xi}(\hh-\hO)(0)<-\frac{1}{8}c_0,$ for some unit direction $\xi$ in the $x''$-subspace, then the conclusion (2) of Lemma \ref{QuadraticApproxLowerStratum} holds.}

The key observations are that $u-h$ is a subsolution and $D_{ee}u + c_0 \eps$ is a supersolution for the same linearized operator $L_u$, and that the two functions can be compared in the domain $B_1 \cap \{u>0\}$. On the other hand $u-h$ is a global supersolution for $L_h$, and then its minimum in $B_{1/2}$ is controlled below by its value at any given point that is not too close to $\mathcal C_\eta$. The hypothesis at the origin is used to guarantee that this minimum value for $u-h$ in $B_{1/2}$ is at least $\eps^\mu$. Now we provide the details.

By \eqref{hC2} and the fact that $D_{\xi\xi}\hO$ is constant, we conclude $D_{\xi\xi}(\hh-\hO)<-\frac{1}{16}c_0$ in $B_{c}$ for a universal $c>0$. Together with \textit{Step 2}, this implies the existence of some $x^*\in B_{1/4}$ such that $(\hh-\hO)(x^*)<-c$ for some universal $c$, that is, $h(x^*)<-c\eps.$

With the universal Lipschitz regularity of $h$, we get $$h<-c\eps \quad \Longrightarrow \quad u-h\ge c\eps \quad \text{ in $B_{c'\eps}(x^*)$}$$ for some small universal $c,c'>0.$ 

Note that $L_h(u-h)\le 0$ in $B_1$ as in \eqref{CompareWithLinearizedEquation}, $u=h$ on $\partial B_1$. We compare $u-h$ to the corresponding solution of the maximal Pucci operator in $B_1 \setminus B_{c'\eps}(x^*)$ and obtain as a consequence of Harnack inequality 
$$u-h\ge \eps^\mu \text{ in $B_{1/2}$,}$$
for some universal $\mu>1$.
Moreover, since $u-h$ solves a linear equation away from $\mathcal C_\eta$, the same argument combined with Harnack inequality imply that
\begin{equation}\label{u-h2}
u-h\ge c (u-h)(\frac 12 \rho e_1) \ge c \eps^\mu \text{ in $B_{1/2}$.}
\end{equation} 
As in \textit{Step 4} of the proof for Lemma \ref{QuadraticApproxTopStratum}, for $e\in\Sph$, we define $$w=\Dee u+c_0\eps.$$ This is a nonnegative function satisfying $L_u(w)\le 0$ in $B_1\cap\PosS.$ Note that $w\ge c_0\eps$ along $\partial\PosS$, and $2 \eps \ge u-h$ in $B_1$, hence
$$w\ge \frac{c_0}{2}(u-h) \quad \text{along $\partial(B_1\cap\PosS)$.}$$
Since  $L_u(w)\le0\le L_u(u-h)$ in $B_1\cap\PosS$, we have $$w\ge\frac{c_0}{2}(u-h) \text{ in $B_1\cap\PosS.$}$$
Combining this with \eqref{u-h2} we find
$$w\ge c(u-h)(\frac{1}{2}\rho e_1)\quad \text{in $B_{1/2}$,}$$ 
which means 
$$\Dee u\ge -c_0\eps+ c(u-h)(\frac{1}{2}\rho e_1) \quad \text{in $B_{1/2}$.}$$
Define the right-hand side to be $-c_0\eps'$, then $$\eps'=\eps-\frac{c}{c_0}(u-h)(\frac{1}{2}\rho e_1)\le\eps-\eps^{2\mu}.$$ Also, $(u-h)(\frac{1}{2}\rho e_0)=C(\eps-\eps')$ as in the second alternative in Lemma \ref{QuadraticApproxLowerStratum}.
\end{proof}

%%%%%%%%%%%%%%%%%%%%%%%%%%%%%%%%%%%%%%%%%%%%%%%%%%%%%%%%%%%%%%%%%%%%%%%%%%%%%%%%%%%%%%%%%%%%%%%%%%%%%%%%%%%%%%%%%%%%%%%%%%%%%%%%%%%%%%%%%%%%%%%%%%%%%%%%%%%%
\section{Iteration scheme and proof of main result}
Lemma \ref{QuadraticApproxTopStratum} and Lemma \ref{QuadraticApproxLowerStratum} form the basic building blocks of the iteration scheme that we perform to prove the main result. As mentioned in Introduction and at the beginning of Section 4, such iteration scheme compensates the absence of monotonicity formulae. 

In the following proposition, we give the details of this iteration when the approximating polynomial $p$ satisfies $\lambda_2(D^2p)\gg\eps$. Again $\lambda_2(M)$ denotes the second largest eigenvalue of the matrix $M$.  The proposition implies that once this condition is satisfied,  it holds true for all approximating polynomials in the iteration. 

\begin{prop}\label{Iteration}
Suppose $u\in\SPE$ for some $p\in\UQua.$ There are universal constants $\bar{\eps}, c>0$ small and $\kappa$, $C$ large, such that if $\eps<\bar{\eps}$ and $\lambda_2(D^2p)\ge\kappa\eps$, then there is $q\in\Qua$ with $|D^2q-D^2p|<C\eps$ such that $$|u-q|(x)\le C|x|^2|\log|x||^{-c} \text{ in $B_{1/2}$.}$$
\end{prop} 

\begin{proof}We will take $\bar{\eps}$ as in Lemma \ref{QuadraticApproxLowerStratum}, and take $\kappa$ to be much larger than $\kappa_0$ as in that lemma.

Define $u_0=u$, $p_0=p$, and $\eps_0=\eps$. Let $h_0$ be the solution to $$\begin{cases}F(D^2h_0)=1 &\text{ in $B_1$,}\\h_0=u_0 &\text{ along $\partial B_1$.}\end{cases}$$
We apply Lemma \ref{QuadraticApproxLowerStratum} to get a $p'\in\UQua$ and $\eps'$ such that $u_0\in\mathcal{S}(p',\eps',\rho)$.   

In general, once $u_k$, $p_k$, $\eps_k$ are found satisfying $u_k\in\mathcal{S}(p_k,\eps_k,1)$ with $\eps_k<\bar{\eps}$ and $\lambda_2(D^2p_k)\ge\kappa_0\eps$ as in Lemma \ref{QuadraticApproxLowerStratum},   we apply that lemma to get $p'\in\UQua$ such that $u_k\in\mathcal{S}(p',\eps',\rho)$. Then we update and define $\eps_{k+1}=\eps'$, $$u_{k+1}(x)=\frac{1}{\rho^2}u_k(\rho x),$$ and $$p_{k+1}(x)=\frac{1}{2}x\cdot D^2p' x+\frac{1}{\rho}\nabla p'(0)\cdot x.$$
This gives $u_{k+1}\in\mathcal{S}(p_{k+1},\eps_{k+1},1).$ And we solve 
 $$\begin{cases}F(D^2h_{k+1})=1 &\text{ in $B_1$,}\\h_{k+1}=u_{k+1} &\text{ along $\partial B_1$}\end{cases}$$ to get $h_{k+1}.$

In particular, Proposition \ref{EstimateForDifference} gives \begin{equation}\label{last}
|D^2h_{k+1}(0)-D^2p_{k+1}|\le C\eps_{k+1}
\end{equation} for universal $C$.

This is called a \textit{step} of the iteration.  

Suppose the assumptions in Lemma \ref{QuadraticApproxLowerStratum} are always satisfied at each \textit{step}. Then we get a sequence of $\{\eps_k\}$. We then divide all steps into different \textit{stages} depending on $\eps_k$.

The $0$th \textit{stage} begins at the $0$th step, and terminates at step $k_0$ if $\eps_{k}\ge(1-\beta)\eps_0$ for all $k\le k_0$ and $$\eps_{k_0+1}<(1-\beta)\eps_0,$$ where $\beta$ is the constant in Lemma \ref{QuadraticApproxLowerStratum}. Then we define $$\eps^{(1)}=\eps_{k_0+1}.$$

The $1$st \textit{stage} begins with step $k_0+1$, and terminates at step $k_1$ if $\eps_k\ge(1-\beta)\eps^{(1)}$ for all $k\le k_1$ and $$\eps_{k_1+1}<(1-\beta)\eps^{(1)}.$$ Then we define $\eps^{(2)}=\eps_{k_1+1}$ and begins the second stage. 

In general, once we have $\eps^{(s)}=\eps_{k_{s-1}+1}$ and begin the $s$th stage at step $ k_{s-1}+1$, this stage terminates at step $k_s$ if  $\eps_k\ge(1-\beta)\eps^{(s)}$ for all $k\le k_s$ and $$\eps_{k_s+1}<(1-\beta)\eps^{(s)}.$$ Then we define $\eps^{(s+1)}=\eps_{k_s+1}$ and begin the $(s+1)$th stage.

Note that within the same stage, each step falls into alternative (2) in Lemma \ref{QuadraticApproxLowerStratum}. Also, within the $s$th stage, $$\eps_{k+1}\le \eps_{k}-\eps_k^\mu\le \eps_k-(1-\beta)^\mu(\eps^{(s)})^\mu.$$   In particular, each stage terminates within finite steps.

Also, suppose $k$ and $k+1$ are two steps within the same stage. Define $$\tilde{h}_{k}(x)=\frac{1}{\rho^2}h_{k}(\rho x).$$Then by definition $\tilde{h}_{k}(x)\le \frac{1}{\rho^2}u_{k}(\rho x)=h_{k+1}(x)$ along $\partial B_1.$ 

Thus $\tilde{h}_k\le h_{k+1}$ in $B_1.$

Meanwhile, since within each stage, each iteration falls into alternative (2) as in Lemma \ref{QuadraticApproxLowerStratum}, we have \begin{align*}(h_{k+1}-\tilde{h}_k)(\frac{1}{2}e_1)&\le (u_{k+1}-\tilde{h}_k)(\frac{1}{2}e_1)\\&=\frac{1}{\rho^2}(u_k-h_k)(\frac{1}{2}\rho e_1)\\&\le C(\eps_k-\eps_{k+1}).\end{align*} Consequently, with Harnack inequality and Proposition \ref{EstimateForDifference}, we get \begin{equation}\label{Twoh}
|D^2h_{k+1}(0)-D^2h_k(0)|\le C(\eps_k-\eps_{k+1})
\end{equation} for a universal $C$. 

Now we focus on the $s$th stage, which consists of steps $\{k_s,k_s+1,k_s+2,\dots\}$, and $\eps^{(s)}=\eps_{k_s}<\bar{\eps}.$

Suppose the first approximating polynomial in this stage,  $p_{k_s}$, satisfies \begin{equation}\label{Condition1}
\lambda_2(p_{k_s})\ge(\kappa_0+A)\eps_{k_s},
\end{equation} where $\kappa_0$ is the constant as in Lemma \ref{QuadraticApproxLowerStratum} and $A$ is a universal constant to be chosen. 

Then for any $k$th step within this stage, we can apply \eqref{last} and \eqref{Twoh} to get \begin{align*}|D^2p_k-D^2p_{k_s}|&=|D^2p_k-D^2h_k(0)|+|D^2p_{k_s}-D^2h_{k_s}(0)|+|D^2h_{k}(0)-D^2h_{k-1}(0)|\\&+|D^2h_{k-1}(0)-D^2h_{k-2}(0)|+\dots+|D^2h_{k_{s}+1}(0)-D^2h_{k_s}(0)|\\&\le C\eps_k+C\eps_{k_s}+C(\eps_{k}-\eps_{k-1}+\eps_{k-1}+\eps_{k-2}-\dots+\eps_{k_s}-\eps_{k_s+1})\\&\le C\eps_{k_s}\end{align*}for a universal $C$.

Consequently, $$\lambda_2(D^2p_k)\ge(\kappa_0+A-C)\eps_{k_s}\ge(\kappa_0+A-C)\eps_k.$$ If $A$ is chosen universally large, then $$\lambda_2(D^2p_k)\ge\kappa_0\eps_k.$$ Consequently, the assumptions in Lemma \ref{QuadraticApproxLowerStratum} are always satisfied. 

Note that the same estimate give $$|D^2p_n-D^2p_m|\le C\eps_m$$ for $n\ge m$ in the same stage. 

Now for general $n$, suppose $n$ is in the $s$th stage. Suppose $k_j$ is the starting step in the $j$th stage, then \begin{align*}
|D^2p_n-D^2p_0|&\le |D^2p_n-D^2p_{k_s}|+|D^2p_{k_s}-D^2p_{k_{s-1}}|+\dots+|D^2p_{k_{1}}-D^2p_0|\\&\le C\eps^{(s)}+C\eps^{(s-1)}+\dots+C\eps^{(0)}\\&\le C\sum_{j=1}^s(1-\beta)^j\eps_0
\end{align*}since between different stages, $\eps^{(i+1)}\le(1-\beta)\eps^{(i)}.$

Therefore, $$|D^2p_n-D^2p_0|\le C\eps_0$$ for some universal $C$ for all $n$.  Consequently, if we choose $\kappa$ universally large, then $$\lambda_2(D^2p_n)\ge\lambda_2(D^2p_0)-C\eps_0\ge(\kappa-C)\eps_0\ge \kappa_0\eps_n$$ for all $n$, and the assumptions in Lemma \ref{QuadraticApproxLowerStratum} are always satisfied. 

Similar estimate gives $$|D^2p_n-D^2p_m|\le C\eps_m$$ whenever $n\ge m.$ 

As a result, there will be a quadratic polynomial $q$ defined as $$q(x)=\frac{1}{2}x\cdot Mx,$$ where $D^2p_n\to M$ and $|M-D^2p_n|\le C\eps_n.$

Note in particular that $q\in\Qua$ and $|D^2q-D^2p|\le C\eps.$

Inside $B_{\rho^m}$, 
$$|u-q-\nabla p_m(0)\cdot x|\le C\eps_m\rho^{2m}.$$This implies $|\nabla p_{m+1}(0)-\nabla p_{m}(0)|\le C\eps_m\rho^{m-1}.$ In particular, $\nabla p_m(0)\to b$ for some $b\in\R^d$ as $m\to\infty.$ By passing to the limit in the previous estimate, we have $$|u-q-b\cdot x|\le C\eps_m\rho^{2m} \text{ in $B_{\rho^m}.$}$$ This forces $b=0.$

Consequently, $|u-q|\le C\eps_m\rho^{2m}$ in $B_{\rho^m}.$ 

Combined with $\eps_m\le\eps_{m-1}-\eps_{m-1}^\mu$, we have the desired estimate, where $c$ depends on $\mu$. 
\end{proof} 

Now the parameter $\kappa$ is fixed depending on universal constants, constants in Lemma \ref{QuadraticApproxTopStratum} become universal. See Remark \ref{Remark}.

We can now give a full description of the iteration scheme:

For $u\in\SPE$ with $\eps<\bar{\eps}$, where $\bar{\eps}$ is the smaller constants between the ones in Lemma \ref{QuadraticApproxTopStratum} and Proposition \ref{Iteration}. 

Define $u_0=u$, $p_0=p$ and $\eps_0=\eps$. Once we have $u_k\in\mathcal{S}(p_k,\eps_k,1)$ we apply Lemma \ref{QuadraticApproxTopStratum} or Lemma \ref{QuadraticApproxLowerStratum}, depending on the comparison between  $\lambda_2(D^2p_k)$ and $\kappa\eps_k$,  to get $p_{k+1}$ such that $$u_k\in\mathcal{S}(p_{k+1},\eps',r_k).$$ Here $r_k=\rho$ if we are applying Lemma \ref{QuadraticApproxLowerStratum}, and $r_k\in(\bar{r},1/2)$ if we are applying Lemma \ref{QuadraticApproxTopStratum}.

Then we define $u_{k+1}(x)=\frac{1}{r_k^2}u_k(r_kx)$ and $\eps_{k+1}=\eps'$. This gives $$u_{k+1}\in\mathcal{S}(p_{k+1},\eps_{k+1},1)$$ and completes a generic step in this iteration. 

If there is some $k_0$ such that $\lambda_2(D^2p_{k_0})\ge\kappa\eps_{k_0}$, then we apply Proposition \ref{Iteration} to see that similar comparison holds for all $p_k$, $k\ge k_0.$ This gives a polynomial $q$ with $\lambda_2(D^2q)\ge C\eps_0$ such that \begin{equation}\label{Rate1}|u-q|\le C|x|^{2}|\log|x||^{-c} \text{ in $B_{1/2}$.}\end{equation}

If for all $k$, $\lambda_2(D^2p_k)\le\kappa\eps_k$, then we are always in the case described by Lemma \ref{QuadraticApproxTopStratum}, where each time the improvement is $\eps\to(1-\beta)\eps$. Here standard argument gives a polynomial $q\in\Qua$ such that 
\begin{equation}\label{Rate2}|u-q|\le C|x|^{2+\alpha} \text{ in $B_{1/2}$}\end{equation} for a universal $\alpha\in(0,1).$ Moreover, in this case, we have $\lambda_2(D^2q)=0.$

Once we have the explicit rates of approximation as in \eqref{Rate1} and \eqref{Rate2}, it is standard that we have the uniqueness of blow-up:

\begin{thm}\label{UniquenessOfBlowUp}
Suppose $u$ solves \eqref{OP} and $x_0\in\Singu$. Then there is a unique quadratic solution, denoted by $p_{x_0}\in\Qua$, such that $\frac{1}{r^2}u(x_0+r\cdot)\to p_{x_0}$ locally uniformly in $\R^d$ as $r\to0.$
\end{thm} 

In particular, these is no ambiguity in the following definition of the strata of the singular set:
\begin{defi}\label{Strata}
Suppose $u$ solves \eqref{OP}. For an integer $k\in\{0,1,\dots,d-1\}$, the $k$th stratum of the singular set $\Singu$ is defined as $$\Sigma^k(u)=\{x\in\Singu:dim(ker(D^2p_x))=k\}.$$
\end{defi} Here $p_x$ is the blow-up profile at point $x$ as in Theorem \ref{UniquenessOfBlowUp}.

Now we can give the proof of the main result:
\begin{proof}[Proof of Theorem \ref{MainResult}]

Let $K\subset\Omega$ be a compact set. With Proposition \ref{UniformApproximation} and Proposition \ref{UniformConvexity} we know that there is $r_K>0$, such that for any $x_0\in\Singu\cap K$ $$|u(x_0+\cdot)-p|\le\bar{\eps}r_K^2 \text{ in $B_{r_K}$}$$ for some $p\in\Qua,$  and $$D^2u(x_0+\cdot)\ge-c_0 \bar{\eps} \, I\text{ in $B_{r_K}$.}$$ 

Define $\tilde{u}(x)=\frac{1}{r_K^2}u(x_0+r_Kx)$, then we  start the iteration as described before Theorem \ref{UniquenessOfBlowUp}.

We have that $x_0\in\Sigma^{d-1}(u)$ if and only if $\lambda_2(D^2p_k)\le\kappa\eps_k$ for all $k$ in the iteration. In this case we have $$|\tilde{u}-p_{x_0}|\le C|x|^{2+\alpha} \text{ in $B_{1/2}$.}$$ Scaling back, we have $$|u(x_0+\cdot)-p_{x_0}|\le C|x|^{2+\alpha} \text{ in $B_{\frac{1}{2}r_K}$}$$ for some $C$ depending on $r_K$ but nevertheless uniform on the set $K$.

After this, it is standard to apply Whitney's extension theorem and get the $C^{1,\alpha}$-covering of $\Sigma^{d-1}(u)\cap K.$ For details of this argument, see Theorem 7.9 in Petrosyan-Shahgholian-Uraltseva \cite{PSU}.

Similar argument works for $x_0\in\Sigma^k(u)$ for $k=1,2,\dots,d-2.$ Instead of \eqref{Rate2}, we have \eqref{Rate1}, which gives the $C^{1,\log^{c}}$ regularity of covering for lower strata. For the details see \cite{FSe}. \end{proof}

%%%%%%%%%%%%%%%%%%%%%%%%%%%%%%%%%%%%%%%%%%%%%%%%%%%%%%%%%%%%%%%%%%%%%%%%%%%%%%%%%%%%%%%%%%%%%%%%%%%%%%%%


\begin{thebibliography}{100}
\bibitem[CSS]{CSS} L. Caffarelli, S. Salsa, L. Silvestre, {\em Regularity estimates for the solution and the free boundary of the \OP for the fractional Laplacian}, Invent. Math. 171 (2008), 425-461.
\end{thebibliography}
\end{document}